\documentclass[preprint,11pt 
]{hjm1}

\usepackage[all]{xy}
\usepackage{subfig}
\usepackage{amsthm,amsfonts,amssymb,amscd,amsmath,enumerate,verbatim,calc,graphicx,geometry}

\allowdisplaybreaks[1]

\newtheorem{theorem}{Theorem}[section]
\newtheorem{lemma}[theorem]{Lemma}
\newtheorem{proposition}[theorem]{Proposition}
\newtheorem{corollary}[theorem]{Corollary}
\theoremstyle{definition}
\newtheorem{definition}[theorem]{Definition}
\newtheorem{remark}[theorem]{Remark}
\newtheorem{example}[theorem]{Example}

\usepackage{tikz}
\usepackage{geometry}[]

\begin{document}

\title{On Targeted Complexity of Discrete $m$-step Motion}

\author[A. Babaee]{Ameneh Babaee$^1$}
\author[H. Mirebrahimi]{Hanieh Mirebrahimi$^{2*}$}\thanks{*Corresponding author}
\author[H. Torabi]{Hamid Torabi$^3$}
\author[S. Fahimi]{Soheila Fahimi$^4$}
\address{$^{1,2,3,4}$ Department of Pure Mathematics, Faculty of Mathematical Sciences, Ferdowsi University of Mashhad, P. O. Box 1159, Mashhad 91775, Iran.}
\email{ambabaee057@gmail.com}
\email{h\_mirebrahimi@um.ac.ir}
\email{h.torabi@ferdowsi.um.ac.ir}
\email{fahimi1440@gmail.com}




\begin{abstract}
In this paper, we study 
targeted simplicial complexity $TC(K, L)$ introduced for situations where the configuration space possesses a simplicial structure $K$ together with a set of configurations $L$ as the target of motion. This type of complexity admits smaller values than the discrete version $TC(K)$. 
We then demonstrate that targeted simplicial complexity  is strongly homotopy invariant and it varies between  simplicial LS-categories of $K$ and $K \prod K$. 
Utilizing this information, we calculate targeted simplicial complexity for cases such as strongly collapsible complexes being equal to zero and for categoriacl subcomplex $L$, $TC(K,L) = scat(K)$. 
Moreover, we compare targeted simplicial complexity with relative topological complexity getting $TC(|K|, |L|) \le TC (K,L)$ where $|\cdot|$ denotes the geometric realization functor, and they are equal in certain cases, such as arbitrary wedges of triangulated circles. 
Also we define targeted $m$-step simplicial complexity of  motions $TC_m(K,L)$ by using $m$-paths, paths whose length is smaller than or equal to $m$, to solve the problems of motion where the robot needs to be charged or repaired after $m$-steps.
For $m$-step simplicial complexity a new invariance holds, which we call $m$-homotopy invariance introduced by $m$-paths.
Finally we compare targeted $m$-step simplicial complexity with $m$-simplicial category $Scat_m$ to obtain some lower and upper bounds and then we prove the sequence of inequalities $scat_{m}(K)\leq TC_{m}(K,L) \le TC_{m}(K) \leq scat_{[\frac{m}{2}]}(K\prod K)$.
\end{abstract}



\maketitle


\section{Introduction and Motivation}
Today, industries, medicine, and even daily life have become almost impossible without the use of robots. Some tasks are performed with robots guided by humans, but the main advantage lies in using automatic robots, which save time and manpower without direct human guidance. An automatic robot is a mechanical device pre-programmed to move without human intervention. These robots require programming to determine their movements from one state to another, leading to foundation of the motion planning problem in robotics.

This article presents approaches to solve the motion planning problem, building on Farber's work. In 2003, Farber introduced the concept of topological complexity to optimize robot motion by stabilizing motion rhythm \cite{Far}. Topological complexity is a numerical value indicating the minimum number of programs needed to move through a configuration space. As topological complexity increases, the number of required programs also increases. Farber chose this term because he viewed the configuration space as a general topological space, with zero topological complexity in shrinkable spaces and increasing complexity as the space moves away from contractibility.
Further studies on topological complexity explored applications such as complexity of maps \cite{Pav}, complexity of projective spaces \cite{Farproj}, and combinatorial complexity \cite{Combi}. Various versions of topological complexity, including simplicial complexity for simplicial complexes as a generalization of graphs, were introduced in parallel studies.

Although the configuration space is generally assumed to be a topological space, studying configuration spaces only with topological tools seems to be a complicated and difficult process. Therefore, in many cases, researchers investigated specific spaces as configuration spaces. Reference \cite{Mes} studied manifolds as  configuration spaces. However, in more complicated cases, path planner algorithms consider the robot's motion point-by-point and discrete which can be modelled explicitly by graphs; The points are the vertices of graph  connected by edges. In motion planning algorithms, programmers often analyse configuration space using graph geometry. Among the common traditional algorithms that utilize graphs for planning robot movement are road map and cell analysis algorithms \cite{Teja}. In modern algorithms, such as genetic algorithms \cite{Li} and ant colony algorithms, the focus is on the optimization, but the geometric of configuration space of robots is still using graphs. It appears that except for cases involving robots moving continuously \cite{5}, the motion planning for other robots can be solved by using graph geometry with less complexity.
On the other hand, the combinatorial approach in algebraic topology is applicable and beneficial in many fields. For instance, combinatorial topology has applications in various areas of computer science, including distributed computing, sensor networks, semantics of concurrency, robotics, and vision. Therefore, the use of simplicial complexes is significant for some topologists, particularly in the field of robotics.

Several viewpoints of the combinatorial version of topological complexity have been introduced, some of which we mention: In \cite{Combi}, topological complexity for finite partially ordered sets was presented using the concept of combinatorial path, termed combinatorial complexity (CC). One of the main results in this article was that the combinatorial complexity of a finite space $P$ serves as an upper bound for the topological complexity of the order complex $\kappa (P)$. Additionally, in \cite{Simp}, Gonzales adapted Farber's topological complexity to the domain of simplicial complexes (SC) by employing the notion of the barycentric subdivision functor and the direct product $K\times K$. Another discrete version of TC was established in \cite{dis}, focusing on finite simplicial complexes. It is important to note that the perspectives presented in \cite{dis} and \cite{Simp} differ.
Nonetheless, Fernandez-Ternero and others \cite{sim} introduced an invariant definition of discrete topological complexity using the concepts of simplicial fibration and the path complex $PK$.

Note that other discrete versions were also introduced and studied, but these perspectives did not incorporate simplicial complexes, namely discrete topological complexity introduced in \cite{Has}.
 In this paper, we extend the approach of \cite{dis} with a twist, considering the robot's movement with specific targets. Targeted motion of the robot optimizes the number of programs and rules because it eliminates the need for planning the useless movements. 
This optimization aids in routing algorithms, such as genetic algorithms, which compare existing routes to find the most optimal route in terms of length and travel time. When the number of routes is reduced, the algorithm requires less time and operations for comparison. The concept of targeted movement was developed based on Short's idea in reference \cite{rel}, which assumed a condition at the path's endpoint for motion planning. Subsequently, in the article \cite{aghil}, the targeted movement of the robot was explored with the condition that the path's endpoint resides in a specific set, focusing on task mappings. In this article, we analyse and discuss the robot's targeted motion in the configuration space utilizing the properties of simplicial complexes.
In the optimization of router algorithms, parameters such as route length, time, cost, etc., are emphasized. However, in the examination of topological complexity and simplicity, the focus is on maintaining the robot's movement rhythm and stability. This is crucial because abrupt and frequent changes in movement increase the risk of robot damage. Moreover, unstable movements can disrupt task execution. Hence, recognizing the significance of motion stability and applying simplicial complexity in robot programming, we introduce the concept of targeted simplicial complexity of the robot within the space modelled by simplicial complexes. To achieve this, we designate a simplicial complex like $K$ and a subcomplex like $L$ as the target complex and define the targeted simplicial complexity denoted by the symbol $TC(K,L)$.

Another problem investigated in this paper is the limitation of power a robot need to move.
This is a common problem and usual challenge for programming a robot to work. Then all types of complexity not depending on path's length may not be enough to program these robots. Especially if the task the robot is doing requires a lot of power. Although with the advancement of technology and the application of energy storage knowledge, the need to charge the robot during the task will be reduced, but this need will never be eliminated, except for robots that are constantly connected to a power source. 
In this paper we introduce the $m$-step complexity on simplicial complexes to present an approach to study  cases where robots must be charged,  repaired or maintained after $m$ steps of motion.
The $m$-step complexity can also plan the movement of robots that have a time limit on their charging, i.e., after $m$ seconds, the robot's charging ends and the robot stops moving or performing work. In this case, too, the  $m$-step complexity can be used to plan the robot's movement by matching the movement and time.

The paper is organized as follows: In section 2, 
we recall the preliminaries needed. Some basic notions about topological complexity, sectional categories, and simplicial complexes are mentioned. Additionally, we review two homotopy relations between simplicial maps called p-homotopy and contiguity. Finally, some results about discrete topological complexity are collected.

In Section 3, we define $TC(K,L)$ for a pair of simplicial complexes $(K, L)$. 
Then we describe motion planning on a simplicial complex $K$ in the targeted motion. We present an equivalent definition of relative topological complexity $TC(X,Y)$ defined in \cite{rel}. 
Then by a similar procedure, we offer the relation between the path complex $P(K,L)$ and $TC(K,L)$ for the pair $(K,L)$. 
This relation shows that $TC(K,L)$ is the number of continuous motion planners in $K$ such that the endpoint of the robot's motion falls within the subcomplex $L$. Hence, we consider the subcomplex $L$ as the target of the robot's motion in $K$ and call $TC(K,L)$ the targeted simplicial complexity.
 Farber \cite{Far} proved that $TC(X)$ equals the homotopy sectional category $hsecat$ of a path fibration defined from $P(X)$ onto $X \times X$ for a path-connected space $X$. 
For the notion of targeted simplicial complexity, the equality $TC(K,L)=hsecat(\pi)$ holds if $\pi :P(K,L)\to K\prod  L$ is the simplicial fibration map defined by $\pi (\gamma )=(\alpha_1 (\gamma ), \omega_1 (\gamma ))$ where $\alpha_1$ and $\omega_1$ denote the initial and terminal simplicial map respectively. Finally, we compare $TC(K,L)$ with the invariant $TC(K)$ introduced in \cite{dis}.

In section 4, we prove that our notion $TC(K,L)$ is an invariant of the strong homotopy type for pairs and then we compare $TC(K,L)$ with the LS-category of $K$ denoted by $scat(K)$ which is defined in \cite{Fer}. More precisely, the inequalities $scat(K) \le TC(K,L) \le TC(K) \le scat (K \prod K)$ hold whenever $K \prod K$ is the categorical product of two copies of $K$. By this fact, it is shown that $TC(K,L) =0$ if and only if $K$ is strongly collapsible; the simplicial version of strongly contractible. Moreover, if the target set is strongly collapsible, then $TC(K,L) = scat(K)$.

In section 5, we compare our new invariant $TC(K,L)$ with the topological complexity of the geometric realization of the pair $(K,L)$ and we present the inequality $TC(|K|,|L|) \le TC(K,L)$. We do this by a theorem from \cite{Int}, which presents a metric on the space $P(X,Y)$. At the end, we show that if the given space can be triangulated as the wedge of some circles, then the equality $TC(|K|,|L|) = TC(K,L)$ holds.

In Section 6, we introduce $m$-step complexity for simplicial complexes by $m$-path's known as path's with length $m$. Then we study the targeted version of this complexity as a generalization of the targeted simplicial complexity defined in Section 3. In the rest of the section, we compare the quantities of this family of complexities. Also to enhance the usage, we investigate homotopy invariance of $m$-step complexity with respect to $m$-homotopy equivalence, a modified version of simplicial homotopy made by $m$-path's. By the same approach, we define $m$-simplicial category and compare it with the $m$-step complexity.

\section{Preliminaries}

In this section we intend to list the prerequisites for the upcoming sections, which are essential for our studies.
Topological complexity was introduced by Farber in \cite{Far} to solve the motion planning problem of mechanical robots. According to this concept, we can find a motion planning algorithm with continuous rules that corresponds a path from point $A$ to point $B$ for every pair of configurations $(A, B)$ in a given topological space $X$. 
A useful tool for studying topological complexity is the concept of the \v{S}varc genus of a map $f:X\to Y$, denoted by $secat(f)$ and defined below.

\begin{definition} 
Let $f:X \to Y$ be a map. The \v{S}varc genus of $f$ is the  minimum number $n\geq 0$ so that we can cover the codomain $Y$ by open sets $V_0,...,V_n$ and for each $j=0,...,n$ there exists a continuous map $s_j:V_j\to X$ such that $f\circ s_j=i_j$, where $i_j:V_j\hookrightarrow Y$ is the inclusion map. 
If the homotopy relation $f\circ s_j\simeq i_j$ holds instead of the equality $f\circ s_j=i_j$, the number $n$ is called homotopy  \v{S}varc genus of $f$ which is denoted by $hsecat(f)$. 
\end{definition}

Note that if the map $f$ is a fibration, then $hsecat(f)=secat(f)$; See \cite{Sch}. 
Topological complexity of the space $X$, denoted by $TC(X)$, was defined as the \v{S}varc genus of the path fibration $p:PX\to X\times X$ defined  by the rule  $p(\gamma )=(\gamma (0),\gamma (1))$. Here $PX$ denotes the space of all paths in $X$ equipped with the compact-open topology. In \cite{dis}, an equivalent definition of $TC(X)$ was presented which is obtained by the fact that the fibration $p$ and the diagonal map $\Delta _X:X\to X\times X$ have the same homotopy \v{S}varc genus.  
In \cite{rel}, Short introduced the concept of relative topological complexity denoted by $TC(X,Y)$ for the pair of spaces $(X,Y)$. This invariant is used to develop a continuous motion planning algorithm on $X$ where the paths must terminate in a specified subset $Y\subseteq X$.

\begin{definition}[\cite{rel}]\label{de2.2n}
Let $X$ be a configuration space and $Y \subseteq X$. The relative topological complexity of pair $(X, Y)$ is the \v{S}varc genus of fibration map $\pi:P(X,Y)\to X\times Y$ defined by $\pi (\gamma )=(\gamma (0),\gamma (1))$ where $P(X,Y)$ denotes the space of all the paths in $X$ which end in the subspace $Y$.
\end{definition}

In this paper, our goal is to develop this notion into a discrete version where the configuration space is a simplicial complex. We restrict the endpoints to a subcomplex of the simplicial complex $K$ in order to optimize the number of rules of $TC(K)$ and avoid extra computations. 
For example, let $K$ and $L\subseteq K$ be as follows. In \cite{dis} it was proved that $TC(K)=2$, but as a result of Theorem \ref{intr}, we have $TC(K,L)=1$.

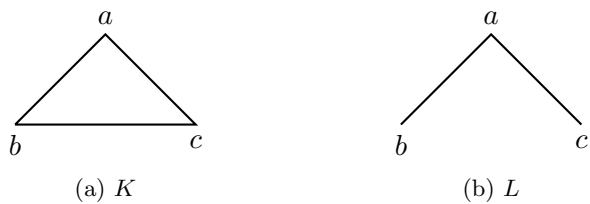
\begin{figure}[!ht]
\centering
\parbox{5cm}{
\subfloat[$K$]
{\begin{tikzpicture}[scale=0.6,	 thick,simple/.style={ball color=black,inner sep=1cm,circle,color=black,text=black, minimum size=0mm},‌
Ledge/.style={to path={
.. controls +(0:1) and +(45:1) .. (\tikztotarget) \tikztonodes}}
]
\node[above]  (a) at ( 2,2){$a$};
‌\node[below] (b) at ( 0,0){$b$};‌
\node[below] (c) at ( 4,0){$c$};‌‌‌
\draw(0,0)--(2,2)--(4,0)--(0,0);
\end{tikzpicture}}}
\parbox{2.5cm}{
\subfloat[$L$]
{\begin{tikzpicture}[scale=0.6,  thick,simple/.style={ball color=black,inner sep=1cm,circle,color=black,text=black, minimum size=0mm},‌
Ledge/.style={to path={
.. controls +(0:1) and +(45:1) .. (\tikztotarget) \tikztonodes}}
]
\node[above]  (a) at ( 2,2){$a$};
‌\node[below] (b) at ( 0,0){$b$};‌
\node[below] (c) at ( 4,0){$c$};‌‌‌
\draw(0,0)--(2,2)--(4,0);
\end{tikzpicture}}}
\end{figure}

We study simplicial complexes as configuration spaces, and we require some tools such as homotopy, sectional category, homotopy equivalence, and other concepts in simplicial modification. First, we review the concepts of contiguity and strong collapse as discussed in \cite{Fer}. These notions are used to define homotopy and contractibility in the context of simplicial literature.

\begin{definition}
Let $K$ and $L$ be two simplicial complexes. 
Two simplicial maps $\phi ,\psi :K\to L$ are contiguous, denoted by $\phi \backsim _c \psi$, if for any simplex $\sigma \in K$, the set $\phi (\sigma )\cup \psi (\sigma )$ is a simplex of $L$. Two simplicial maps $\phi ,\psi :K\to L$ are in the same contiguity class, denoted by $\phi \thicksim \psi$, if there is a sequence
of contiguous simplicial maps $\phi _i:K\to L$, $0\leq i\leq n$, such that
$
\phi =\phi _0 \backsim _c ... \backsim _c \phi _n=\psi
$.
\end{definition}

The contiguity relation between simplicial maps is the simplicial version of homotopy between continuous maps. It can be applied for studying simplicial complexes as configuration spaces. 
In order to define the topological complexity for simplicial complexes, as studied in \cite{dis}, some other basic notions need to be recalled, such as categorical product of simplicial complexes.
The categorical product $K\prod  L$, is a simplicial complex defined below, for any two complexes $K$ and $L$.
 \begin{enumerate}
 \item[$\bullet$]
The set of vertices is defined as $V(K\prod  L):=V(K)\times V(L)$.
 \item[$\bullet$]
For any simplex $\sigma$, $\sigma \in K\prod  L$ if and only if $p_1 (\sigma )\in K$ and $p_2 (\sigma )\in L$, where  $p_1 :V(K\prod  L)\to V(K)$ and $p_2 :V(K\prod  L)\to V(L)$ are the corresponding projections.
 \end{enumerate}

\noindent
Here, by $K^2$ we mean the categorical product of two copies of $K$; That is $K^2 = K \prod K$. 
Notice that there is another definition of product of two simplicial complexes $K$ and $L$, called direct product and mostly denoted by $K\times L$; See \cite{dis}. Contiguity class and categorical product  are the main preliminaries needed to define discrete topological complexity mentioned below.

\begin{definition}[\cite{dis}]\label{de2.2}
The discrete topological complexity $TC(K)$ is the least integer $n\geq 0$ such that $K^2$ can be covered by $n+1$ subcomplexes $\Omega _0,...,\Omega _n$ called Farber subcomplexes, and for each $j=0,...,n$ there exists a simplicial map $\sigma :\Omega _j\to K$ so that $\Delta \circ \sigma\sim i_{\Omega _j}$ where $\Delta :K\to K^2$ is the diagonal map and $i_{\Omega _j}:\Omega _j\to K^2$ is the inclusion map.
\end{definition}

In Definition \ref{de2.2}, discrete topological complexity is defined as the minimum number of Farber subcomplexes. In \cite{sim}, discrete topological complexity was studied using tools such as fibrations and the \v{S}varc genus in simplicial settings. In this paper, we aim to investigate discrete topological complexity from both perspectives. To do so, we revisit some key concepts, including simplicial fibrations, Moore paths, p-homotopy, \v{S}varc genus, and so on.
For $n\geq 1$, let $I_n$ denote the one-dimensional simplicial complex with vertices labeled as integers $\lbrace 0 ,...,n\rbrace $ and  edges represented by $\lbrace j,j + 1\rbrace $ for $0 \leq j < n$. 

\begin{definition}\label{de2.5}
A map $ p: E \to B $ is a simplicial ﬁbration if for given simplicial maps $H: K \prod  I_m\to B$ and $\phi : K \prod  \lbrace 0\rbrace \to E$ as in the following commutative diagram
\[
\xymatrix{
K\prod  \lbrace 0\rbrace \ar[d]_-{i_0^m} 
\ar[r]^-{\phi} & E \ar[d]^-{p}\\
K\prod  I_m \ar[ur]^-{\widehat{H}} \ar[r]_-{H} & B,\\ 
}
\] 
there exists a simplicial map $ \widehat{H}: K \prod  I_m \to E$ such that $\widehat{H}\circ i_0^m =\phi $ and $p \circ \widehat{H}=H$. \\
If $K$ is finite, then $p$ is called a simplicial finite-fibration.
\end{definition}

Consider the natural triangulation of the real line, denoted by $Z$, whose vertices are all the integers $i\in \mathbb{Z}$ and whose 1-simplices are all the consecutive pairs $\lbrace i,i + 1\rbrace$.
Let $K$ be a simplicial complex. A Moore path in $K$ is a simplicial map $\gamma : Z \to K$
such that there exist integers $ i^-,i^+ \in Z$ satisfying the following two conditions:
\begin{enumerate}
\item[(i)]
$\gamma (i)=\gamma (i^-)$, for all $i\leq i^-$; Put $\gamma ^-:=\max \lbrace i^-:\gamma (i)=\gamma (i^-),$  for all $i\leq i^- \rbrace$,
\item[(ii)]
$\gamma (i)=\gamma (i^+)$, for all $i\geq i^+$; Put $\gamma ^+:=\min \lbrace i^+:\gamma (i)=\gamma (i^+)$,  for all $i\geq i^+ \rbrace$.
\end{enumerate}
The values $\alpha_1 (\gamma):= \gamma (\gamma ^- )$ and $\omega_1 (\gamma ):=\gamma (\gamma ^+)$ are called the initial and final vertex of $\gamma $ respectively. 
Considering this notation, any Moore path $\gamma $ in $K$ can be identified with the restricted simplicial map $\gamma : [\gamma ^-,\gamma ^+] \to K$.
The reverse of Moore path $\gamma$ is defined as $\overline{\gamma } :[-\gamma ^+,-\gamma ^-] \to K$ by
$
\overline{\gamma }(i) =\gamma (-i)
$.
Consider two Moore paths $\gamma $ and $\delta $  in $K$ such that $\omega_1 (\gamma ) = \alpha_1 (\delta )$. The product path $\gamma \ast \delta$ is defined as 
\[
(\gamma \ast \delta )(i)=
\begin{cases}
\gamma (i-\delta ^-), & i\leq \gamma ^+ +\delta ^- \\
\delta (i-\gamma ^+ ), &  i\geq \gamma ^+ +\delta ^- .
\end{cases}
\]

The following definition introduces the concept of path complex denoted by $PK$. 
The path complex is essential in the study of discrete topological complexity.

\begin{definition}[\cite{sim}]
Let $K$ be a simplicial complex. The Moore path complex of $K$, denoted by $PK$, is a full subcomplex of $K^Z$ generated by all Moore paths $\gamma : Z \to K$. The vertices $\lbrace \gamma _0,...,\gamma _p \rbrace \subseteq PK $ defines a simplex in $PK$ if and only if 
$
\lbrace \gamma _0 (i),...,\gamma _p(i),\gamma _0(i+1),...,\gamma _p(i+1) \rbrace
$
is a simplex in $K$ for any integer $i\in \mathbb{Z}$.
\end{definition}

The initial and ﬁnal vertices of any given Moore path $\gamma$ deﬁne simplicial maps $\alpha_1 : PK\to K$ and $\omega_1 : PK \to K$.
Also the simplicial map 
$
p = (\alpha_1 , \omega_1 ):PK\to K\prod  K
$
is a simplicial ﬁnite-ﬁbration (see \cite{sim}).
 The maps $\alpha_1$ and $\omega_1$ lead to the following notion of homotopy.

\begin{definition}[\cite{sim}]
Let $f,g: K \to L$ be simplicial maps. The map $f$ is said to be p-homotopic to $g$, denoted by $f \simeq g$, if there exists a simplicial map $H: K \to PL$ such that $\alpha_1 \circ H = f$ and $\omega_1 \circ H = g$. 
Also $f$ is  a p-homotopy equivalence if there exists a simplicial map $h: L \to K$ such that $h \circ f\simeq id_K$ and $f \circ h \simeq id_L$. 
\end{definition}

The relation of p-homotopy is indeed reflexive and symmetric but not transitive. It is also compatible with both left and right compositions \cite{sim}. Furthermore, p-homotopy is shown to be stronger than contiguity as proven in \cite[Proposition 5.10]{sim} cited below.
\begin{theorem}[\cite{sim}]\label{5.10}
Let $f,g: K \to L$ be simplicial maps. Then 
\begin{enumerate}
\item[(i)]
If $f\sim g$, then $f\simeq g$.
\item[(ii)]
If $K$ is finite and $f\simeq g$, then $f\sim g$.
\end{enumerate}
\end{theorem}

There are some close relations between topological complexity and the concept of \v{S}varc genus as presented by Farber \cite{Far}. To study simplicial complexity we need a modified \v{S}varc genus 
adapted to the simplicial setting.

\begin{definition}[\cite{sim}]\label{genus}
Let $\phi : K \to L$ be a simplicial map. 
\begin{enumerate}
\item
The simplicial \v{S}varc genus of  $\phi : K \to L$, denoted by $secat(\phi)$, is the minimum integer $ n \geq 0$ for which $L$ is the union $L_0 \cup ...\cup L_n$ of $n + 1$ subcomplexes, such that for each $j$ there exists a simplicial section $\sigma _j$ of $\phi $; A simplicial map $\sigma _j : L_j \to K$ such that $\phi \circ \sigma _j$ equals the inclusion map $i_j : L_j \hookrightarrow L$. 

\item
The homotopy simplicial \v{S}varc genus of $\phi : K \to L$, denoted by $hsecat(\phi )$, is the minimum integer $ n \geq 0$ such that $L = L_0 \cup ...\cup L_n$, and for each $j \in \lbrace 0,...,n\rbrace $ there exists an “up to contiguity class” simplicial section $\sigma _j$ of $\phi $; A simplicial map $\sigma _j : L_j \to K$ such that $\phi \circ \sigma _j \sim i_j $. 
\end{enumerate}
\end{definition}

Note that if $\phi \circ \sigma _j = i_j $, then $\phi \circ \sigma _j \sim i_j $ and therefore
$hsecat(\phi)\leq secat (\phi )$. The equality holds for some particular classes of maps such as simplicial fibrations.

\begin{theorem}[\cite{sim}]\label{iii}
Let $p:E\to B$ be a simplicial fibration. Then $hsecat(p)=secat(p)$.
\end{theorem}

In Definition \ref{de2.2}, discrete topological complexity was introduced by the union of subcompexes called Farber subcomplexes. In \cite{sim},  as a main result it was proved that the discrete topological complexity is equal to the \v{S}varc genus of a specific simplicial fibration.

\begin{theorem}[\cite{sim}]\label{b}
Let $K$ be a finite complex. The discrete topological complexity of $K$ equals the \v{S}varc genus of the simplicial fibration $p=(\alpha_1 , \omega_1 ):PK\to K\prod  K$; That is $TC(K)=secat(p)$.
\end{theorem}

Lusternik-Schnirelman category is a valuable tool for studying topological complexity. To gain a better understanding of discrete topological complexity, it would be logical to apply Lusternik-Schnirelman category for simplicial structures as well.

\begin{definition}[\cite{Lus}]\label{de2.12n}
Let $K$ be a simplicial complex. We say that the subcomplex $U\subseteq K$ is categorical if there exists $v\in K$ such that the inclusion map $i:U\hookrightarrow K$ and the constant map $c_v:U\hookrightarrow K$ are in the same contiguity class; $i_U\sim c_v$.
Moreover the simplicial LS-category $scat(K)$ of the simplicial complex
 $K$ is the least integer $n \geq 0$ such that $K$ can be covered by $n + 1$ categorical subcomplexes. 
\end{definition}

It is important to note that a categorical subcomplex might not be  connected and therefore it might not be categorical in itself (see \cite{Fer}).
The following theorem shows that the LS-category of a simplicial complex equals the \v{S}varc genus of a specific fibration map.

\begin{theorem}[\cite{sim}]\label{th2.13n}
Let $K$ be a connected finite complex and $P_0 K=\lbrace \gamma \in PK :\, \alpha_1 (\gamma )=v_0 \rbrace $. Then the simplicial LS-category $scat(K)$ equals the \v{S}varc genus of the simplicial finite-fibration $\omega_1:P_0 K\to K$; That is $scat(K)=secat(\omega_1)$.
\end{theorem}

The path space $P(X,Y)$ is equipped with a compatible metric introduced in \cite{Int}. We use this metric to compare the complexity of the simplicial complex and its geometric realization.

\begin{theorem}[\cite{Int}]\label{Int}
Let $X$ be a compact space and $(Y,d)$ a metric space. Then with the compact-open topology, $ P(X,Y)$ is metrizable and the metric  is given by 
\[
e(f,g)=sup \lbrace  d(f(x),g(x)):x\in X \rbrace \quad f,g\in P(X,Y).
\] 
\end{theorem}

To compare topological and simplicial complexity, we need to recall some propositions, such as Theorem \ref{j} recalled from \cite[Proposition 6.14]{com}. Note that in general, the converse statement of theorem \ref{j} does not hold; See \cite[Section 6.5]{com}. A simplicial complex $K$ is strongly collapsible if it is strongly equivalent to a point, or equivalently, the identity map $id_K$ is in the contiguity class of constant map $c_v: K\to K$ for some vertex $v \in K$. In other words, a complex is strongly collapsible if it is categorical in itself.

\begin{theorem}[\cite{com}]\label{j}
If simplicial complex $K$ is strongly collapsible, then its geometric realization $|K|$ is contractible.
\end{theorem}

\section{Targeted Simplicial Complexity}

R. Short introduced the relative topological complexity $TC(X,Y)$ for the pair $(X,Y)$ in \cite{rel}, which helps  to reduce $TC(X)$. We now aim to modify this invariant in the discrete version of topological complexity.
Consider $L\subseteq K$ and $\Omega \subseteq K\prod L$, together with the inclusion map $i_\Omega: \Omega \to K\prod L$. Additionally, let $\Delta_L: L \to K\prod L$ represent the restriction of  diagonal map, defined as $\Delta_L(v) = (v,v)$.

\begin{definition}\label{def}
The targeted simplicial complexity $TC(K,L)$, is the least integer $n\geq 0$ such that one can cover $K\prod  L$ by $n+1$ simplicial subcomplexes $\Omega _0,...,\Omega _n$, so that for any $i=0,...,n$ there exists a simplicial map $\sigma _i:\Omega _i \subseteq K\prod  L\to L$ with $\Delta _L \circ \sigma _i\sim i_{\Omega _i}$.
\end{definition}

By comparing Definition \ref{def} with the definition of discrete topological complexity, Definition \ref{de2.2n}, one can see the following advantages helping the programmer for planning the algorithm of robot motion:
\begin{enumerate}
\item
The simplicial complex $K \prod L$ is smaller than  the categorical product $K^2 = K \prod K$, and then it is naturally easier to compute and study.

\item
Since the complex $K \prod L$ is smaller than $K^2$, fewer number of subcomplexes satisfying Definition \ref{def} are needed to cover $K \prod L$ instead of $K^2$.

\item
Since the number of $\Omega$'s covering $K \prod L$ is smaller than the subcomplexes covering $K^2$, the number of rules needed for the algorithm planning the motion in $K$ is less than the number of rules required for the discrete case.
\end{enumerate}

In \cite{rel}, topological invariant $TC(X,Y)$ was introduced to solve the following motion planning problem: Imagine a robot with configuration space $X$. The objective is to create a motion planning algorithm on $X$ where the paths terminate in $Y$. Here, we demonstrate that our concept of targeted simplicial complexity can effectively describe  motion planning on a simplicial complex $K$ such that the robot's motion ends in a specified subcomplex $L\subseteq K$. To verify, we consider some subcomplex $\Omega$ satisfying Definition \ref{def}.

\paragraph*{Motion planning}\label{Motion}

Let $\Omega \subseteq K\prod  L$ and let $\sigma :\Omega \to L$ be a map such that $\Delta \circ \sigma \sim i_\Omega $. For each pair of points $x\in K$ and $y\in L$, if  $(x,y)\in \Omega $, then the point $\sigma (x,y)$ is an intermediate point of a path from $x$ to $y$, as shown below.
Since $\Delta \circ \sigma \sim i_\Omega $, there exist $h_1,...,h_{m-1}:\Omega \to K\prod  L$ such that $\Delta \circ \sigma =h_0\sim _c ...\sim _c h_j \sim _c ... \sim _c h_m=i_\Omega $. Put $(x_j,y_j)=h_j(x,y)$, and then $x_m=x$, $y_m=y$ and $x_0=\sigma (x,y)=y_0$. Hence we have the following sequence of points:
\begin{align}\label{1}
x=x_m,...,x_0=\sigma (x,y)=y_0,...,y_m=y.
\end{align}
Note that contiguous relationship implies that two consecutive points in the sequence
(\ref{1}) 
belong to the same simplex. Since $h_j \sim _ch_{j+1}$, the points $h_j(x,y)=(x_j,y_j)$ and $h_{j+1}(x,y)=(x_{j+1},y_{j+1})$ generate a simplex of $K\prod  L$.
Then by definition of the categorical product $K\prod  L$, the points $x_j$ and $x_{j+1}$ generate a simplex of $K$ and  $y_j$ and $y_{j+1}$ generate a simplex of $L$. Therefore we obtain a sequence of points \eqref{1} giving an edge-path on $K$ connecting the point $x$ to $y$.
Note that not only $y$ belongs to $L$ but also all the points $y_0$, ...,$y_m=y$ in sequence (\ref{1}) belong to the subcomplex $L$. 

Now we proceed to adapt the definition of $PK$ and define the map $\pi$.
Let $K$ be a simplicial complex and $L\subseteq K$ be a subcomplex. We define $P(K,L)=\lbrace \gamma \in PK : \omega_1 (\gamma)\in L\rbrace $ and $\pi :P(K,L)\to K\prod  L$ by the rule $\pi (\gamma)=(\alpha_1 (\gamma),\omega (\gamma))$, where $\omega$ is the map obtained from $\omega_1$ whenever the codomain is restricted to $L$. 
Proposition \ref{P} presents a remarkable  equality  for $TC(X,Y)$. This idea helps us to find an equivalent expression for the targeted topological complexity in discrete version $TC(K,L)$. 

\begin{proposition}\label{P}
Let $X$ be a topological space and $Y \subseteq X$.
Then the path fibration map $\pi$ and diagonal map $\Delta _{Y}$ have the same homotopic \v{S}varc genus, and both coincide with the topological complexity of $(X,Y)$; That is
\[
TC(X,Y) =secat(\pi) =hsecat(\pi) = hsecat(\Delta _{Y}).
\]
\end{proposition}

\begin{proof}
Definition \ref{de2.2} implies the first equality $TC(X,Y) =secat(\pi)$. Since $\pi$ is a fibration map 
$secat(\pi) =hsecat(\pi)$.
It remains to show that  $hsecat(\Delta _{Y})=hsecat(\pi)$.  We conclude it from the homotopy equivalence $Y\simeq P(X,Y)$  commuting the following diagram up to homotopy
\[
\xymatrix{
Y \ar@<0.5ex>[r]^-{C} \ar[d]_-{\Delta _{Y}}  &P(X,Y) \ar@<0.5ex>[l]^-{\omega } \ar[dl]^-{\pi} \\
X\times Y
}
\]
where $C(y)=c_y$, the constant path at $y$, and $\omega (\gamma)=\gamma (1)$.
Two maps $C$ and $\omega$ are homotopy equivalences; That is $C:Y\simeq P(X,Y)$ as shown below. We have $\omega \circ C(y)=\omega (c_y)=c_y(1)=y=id_Y(y)$. Now we want to show $C \circ \omega \simeq id_{P(X,Y)}$. For this aim define map $F:P(X,Y)\times I\to P(X,Y)$ by $F(\gamma , t)(s)=\gamma ((1-s)t+s)$. We prove that this map is continuous. Let $(K,U)$ be the set of all paths $\gamma $ in $P(X,Y)$ so that $\gamma (K)\subseteq U$ where $K$ is a compact subset of $I$ and $U$ is an open set in $X$. We have 
\[
F^{-1}(K,U)=\lbrace (\gamma ,t):\gamma ((1-s)t+s)\in U;\forall s\in K \rbrace .
\]
Suppose that $(\gamma_0,l_0)\in F^{-1}(K,U)$, then $\gamma _0((1-s)l_0+s)\in U$ for every $s\in K$. It means that $(1-s)l_0+s\in \gamma _0^{-1}(U)$. We know that $\gamma _0^{-1}(U)$ is an open set in $I$. So for any 
$s\in K$ there exists an open set $V_s$ in $I$ such that $(1-s)l_0+s\in V_s \subseteq \gamma_0^{-1}(U)$. Let $f:I\times I\to I$ be the map defined by $f(s,l)=(1-s)l+s$. For any $s\in K$, we have $(s,l_0)\in f^{-1}(V_s)$ and since $f$ is continuous, there exists an open set $J_s \times J^{'}_s$ in $I\times I$ such that $(s,l_0)\in J_s \times J^{'}_s \subseteq f^{-1}(V_s)$. Now $K\subseteq \bigcup _{s\in K}J_s$ and on the other hand $K$ is compact. Therefore there exist elements $s_1,...,s_m\in K$ such that $K\subseteq \bigcup _{i=1}^{m}J_{s_i}$. Put $J{'}=\bigcap _{i=1}^{m}J{'}_{s_i}$. Then for each $s\in K$ and $l\in J{'}$, we have $(1-s)l+s\in \gamma _0^{-1}(U)$. Consider open interval $I^{'}$ so that $l_0 \in I'$ and $\overline{I^{'}}\subseteq J{'}$. We define a compact set $K{'}$ as fallows
\[
K^{'} =\lbrace (1-r)t^{'} +r: r\in K \text{ and } t^{'} \in \overline{I^{'}}\rbrace .
\]
One can see that $(\gamma _0,l_0)\in (K^{'},U)\times I^{'}$ and $(K^{'},U)\times I^{'}\subseteq F^{-1}(K,U)$.
Also the map $F$ satisfies $F(\gamma ,0)=id_{P(X,Y)}(\gamma )$ and $F(\gamma ,1)=C \circ \omega (\gamma )$.
Now we prove that the diagram above  is commutative up to homotopy. We have $\pi \circ C(y)=\pi (c_y)=(y,y)=\Delta _Y(y)$. It is enough to show that $\Delta _Y \circ \omega \simeq \pi $. Define map $H:P(X,Y)\times I\to X\times Y$ by $H(\gamma , t)=(\gamma (t),\gamma (1))$. This map is continuous, because  its associate equals the inclusion map.  
Also the map $H$ satisfies $H(\gamma , 0)=\pi (\gamma )$ and $H(\gamma ,1)=\Delta _Y\circ \omega (\gamma )$. Since $\omega$ is a homotopy equivalence and the diagram is commutative up to homotopy,  any homotopy section of $\pi$  can be corresponded to a homotopy section of $\Delta_Y$. Again since $C$ is a homotopy equivalence, any homotopy section of $\Delta_Y$  can be corresponded to a homotopy section of $\pi$.
Therefore $hsecat (\pi) = hsecat (\Delta_Y)$. 
\end{proof}

Definition \ref{de2.2n} introduced $TC(X,Y)$ as the homotopy \v{S}varc genus of the path fibration map from the path space $P(X, Y)$ to $X \times Y$.  Then in Theorem  \ref{P}, we prove that $TC(X,Y)$ equals the homotopy \v{S}varc genus of the diagonal map. For the simplicial version, we are going the opposite way. In Definition \ref{def}, we  define $TC(K,L)$ as the homotopy \v{S}varc genus of the diagonal map.
Now we intend to prove that our definition
 of $TC(K,L)$ coincides with the homotopy \v{S}varc genus of the path simplicial fibration $\pi: P(K,L) \to K \prod L$.
This equality helps us to investigate properties of the targeted simplicial complexity. 
 
 \begin{theorem}\label{i}
The path simplicial map $\pi: P(K,L) \to K  \prod L$ and the diagonal map $\Delta _{L}: L \to K \prod L$ have the same homotopy \v{S}varc genus. Moreover, if $K$ is a finite complex and $L$ be a subcomplex of $K$, then
\[
TC(K,L)  = hsecat(\Delta _{K})=hsecat (\pi) = secat(\pi).
\]
\end{theorem} 

\begin{proof}
First we show that the simplicial map $\pi: P(K,L) \to K  \prod L$ is a simplicial fibration. To prove we consider a simplicial complex and a simplicial map equivalent to $P(K,L)$ and $\pi$ respectively. 
Suppose that $p:PK\to K\prod  K$ is the usual simplicial fibration, and let us consider $p^{'}$ as its pullback which is induced by the inclusion map $f:K\prod  L\to K\prod  K$. We have the following commutative diagram where $f^{'}$ and $p^{'}$ are projections.

\[\xymatrix{
(K\prod  L)\prod _{K\prod  K}PK \ar[d]_-{p^{'}} \ar[r]^-{f^{'}} & PK \ar[d]^-{p}\\
K\prod  L  \ar[r]_-{f} & K\prod  K\\ 
}
\]
The following simplicial complex is equivalent to path complex $P(K,L)$;
\begin{align*}
\textstyle(K\prod  L)\prod_{K\prod  K}PK&=\lbrace  ((v,v^{'}),\gamma) \in \textstyle (K\prod  L) \textstyle \prod  PK :p(\gamma)=f(v,v^{'}) \rbrace \\
&=\lbrace ((v,v^{'}),\gamma) \in \textstyle(K\prod  L) \textstyle \prod  PK:(\alpha(\gamma),\omega (\gamma))=(v,v^{'})   \rbrace \\
&=\lbrace ((v,v^{'}),\gamma) \in \textstyle(K\prod  L) \textstyle\prod PK: \alpha(\gamma)=v  \ \text{and} \  \omega (\gamma) =v^{'} \rbrace.
\end{align*}
Hence one can replace the path complex 
$(K\prod  L)\prod_{K\prod  K}PK$ with $P(K,L)$ and 
 the map $\pi$ with the pullback of $p$ induced by the inclusion map $f$; That is $p^{'} =\pi$.
Now by \cite[Proposition 4.2]{sim}, any pullback of a simplicial fibration map is also a simplcial fibration, and then we can conclude that the map $\pi :P(K,L)\to K\prod  L$ is a simplicial fibration.
Therefore, $secat (\pi) = hsecat (\pi)$ by Theorem \ref{iii}.

At the second step by a similar argument as Proposition \ref{P}, we show that $hsecat (\pi) = hsecat (\Delta_L)$. We offer a p-homotopy equivalence 
$L\simeq P(K,L)$ 
such that the following diagram  commutes up to p-homotopy
\[
\xymatrix{
L \ar@<0.5ex>[r]^-{C} \ar[d]_-{\Delta _{L}}  &P(K,L) \ar@<0.5ex>[l]^-{\omega} \ar[dl]^-{\pi} \\
K\prod  L
}
\]
where
$C(v)=c_{v}$
and
$\omega(\gamma)=\gamma (\gamma^+)$.
It suffices to show that $\omega \circ C\simeq id_{L}$ and $C\circ \omega\simeq id_{P(K,L)}$.
Obviously $\omega \circ C = id_L$, because $\omega \circ C(v)=\omega(c_{v})=\omega (c_{v})=v=id_{L}(v)$. We define p-homotopy $H:P(K,L)\to P(P(K,L))$ by 
\begin{equation}\label{eq3.1n}
H(\gamma)(i)(j)=\gamma^{i}(j)=
\begin{cases} 
\gamma (j); & i\leq j\\
\gamma (i); & i\geq j
 \end{cases}.
\end{equation}
First we prove that $H$ is a simplicial map.
Assume that $\lbrace \gamma _{0}, ...,\gamma _{p}  \rbrace\in P(K,L)$ is a simplex in $P(K,L)$. Then by definition of a path complex, the set
$
\lbrace \gamma _{0}(i),...,\gamma _{p}(i),\gamma _{0}(i+1),...,\gamma _{p}(i+1)   \rbrace$
is a simplex in $K$ for any $i\in \mathbb{Z}$.
We have to show that $H(\lbrace \gamma _{0}, ...,\gamma _{p} \rbrace )=\lbrace H(\gamma _{0}),...,H(\gamma _{p}) \rbrace $ is a simplex in $P(P(K,L))$ or equivalently
\[
\lbrace H(\gamma _{0})(i),...,H(\gamma _{p})(i),H(\gamma _{0})(i+1),...,H(\gamma _{p})(i+1)   \rbrace \in P(K,L); \text{  for all } i\in \mathbb{Z}.
\]
By the notation of rule \eqref{eq3.1n}, it is equivalent to show that
$\lbrace {\gamma _{0}}^{i},...,{\gamma _{p}}^{i},{\gamma _{0}}^{i+1},...,{\gamma _{p}}^{i+1} \rbrace \in P(K,L)$  for all $i\in \mathbb{Z}$. By the definition of path complex, it suffices to verify that for any $ j\in \mathbb{Z}$,
the set $S$ considered below is a simplex in $K$
\[
\{ \gamma _{0}^{i}(j),...,{\gamma _{p}}^{i}(j),{\gamma _{0}}^{i+1}(j),...,{\gamma _{p}}^{i+1}(j) ,\gamma _{0}^{i}(j+1),...,{\gamma _{p}}^{i}(j+1),{\gamma _{0}}^{i+1}(j+1),...,{\gamma _{p}}^{i+1}(j+1)\}.
 \]
 Suppose that $j\in \mathbb{Z}$ and $0\leq k\leq p$, There exist two cases:
 \begin{enumerate}
 \item[(i)]
 If $j\geq i$, then we have 
 \begin{align*}
 {\gamma _{k}}^{i}(j)&=\gamma _{k}(j);\quad {\gamma _{k}}^{i}(j+1)=\gamma _{k}(j+1)\\
 {\gamma _{k}}^{i+1}(j)&=\gamma _{k}(j);\quad {\gamma _{k}}^{i+1}(j+1)=\gamma _{k}(j+1),
 \end{align*}
and therefore 
$
S=\lbrace \gamma _{0}(j),...,\gamma _{p}(j),\gamma _{0}(j+1),...,\gamma _{p}(j+1)   \rbrace$ being a simplex as desired.
\item[(ii)]
If $j\leq i$, then we have 
\begin{align*}
 {\gamma _{k}}^{i}(j)&=\gamma _{k}(i);\quad {\gamma _{k}}^{i}(j+1)=\gamma _{k}(i)\\
 {\gamma _{k}}^{i+1}(j)&=\gamma _{k}(i+1);\quad {\gamma _{k}}^{i+1}(j+1)=\gamma _{k}(i+1),
 \end{align*}
and hence 
$
S=\lbrace \gamma _{0}(i),...,\gamma _{p}(i),\gamma _{0}(i+1),...,\gamma _{p}(i+1)   \rbrace$ making a simplex as needed.
 \end{enumerate}
Since $S$ is a simplex in $K$, for two cases verified above, $H$ is a simplicial map.
Now we show that $H$ is a p-homotopy map.
By the definition of Moore path and map $\gamma ^{i}$, we have ${\gamma^{i}}^{-}\geq i$,${\gamma^{i}}^{+}\geq i$, $\gamma ^{+}\leq {\gamma^{i}}^{+}$.
Then for each $\gamma \in P(K,L)$ and $i\in \mathbb{Z}$, we have $\alpha \circ H ( \gamma) (i)=\alpha \circ \gamma ^{i}=\gamma ^{i}({\gamma ^{i}}^{-})=\gamma ^i (i)=\gamma (i) = id_{P(K,L)} (\gamma) (i)$ and $\omega \circ H (\gamma) (i)=\omega \circ \gamma ^{i}=\gamma ^{i}({\gamma^{i}}^{+})=\gamma ({\gamma ^i}^{+})= \gamma (\gamma ^{+})=\omega (\gamma )=C\circ \omega(\gamma )(i)$.\\
Now we check that the diagram is commutative up to p-homotopy. We have 
$\pi \circ C(v)=\pi (c_{v})=(v,v)=\Delta _{L}(v)$. Also define p-homotopy $F:P(K,L)\to P(K\prod L)$ by $F(\gamma )=(\gamma ,C\circ \omega (\gamma))$. This map is 
 simplicial because $\gamma $ and $C \circ \omega (\gamma)$ are simplicial maps. Moreover, for any $\gamma \in P(K,L)$, we have 
$\alpha \circ F(\gamma )=(\alpha (\gamma ), \alpha (C\circ \omega (\gamma )))=(\alpha (\gamma ),\alpha (c_{\omega (\gamma )}))=(\alpha (\gamma ), \omega (\gamma ))=\pi (\gamma)$ and $\omega \circ F(\gamma )=(\omega (\gamma ),\omega (C\circ \omega (\gamma )))=(\omega (\gamma ),\omega (c_{\omega (\gamma )}))=(\omega (\gamma ),\omega (\gamma ))=\Delta _{L} \circ \omega(\gamma )$. 
Therefore $F: \Delta _{L}\circ \omega\simeq \pi$ and then by the commutative diagram $hsecat (\pi) = hsecat (\Delta_L)$.
Since $\pi$ is a simplicial fibration $hsecat (\pi) = secat (\pi)$, and hence we conclude the sequence of equalities.
 \end{proof}

%
%

Note that by the commutative diagram in the proof of Theorem \ref{i}, we can conclude that any motion planner described by Definition \ref{def}, leads to a motion planner using the new definition of $TC(K,L)$ in Theorem \ref{i} and vice versa. 

Short proved that $TC(X,Y)\leq TC(X)$ in \cite{rel}.  
We achieve a similar result for the pair of simplicial complexes $(K,L)$, and then we found $TC(K)$ as an upper bound for $TC(K,L)$ by the proposition stated below. Proposition \ref{a} proves that the \v{S}varc genus of a simplicial fibration  is always greater than  the \v{S}varc genus of any pullback by which induced.

\begin{proposition}\label{a} 
Let $p:L\to K$ be a simplicial fibration and let $f:K^{'}\to K$ be a simplicial map. If $p^{'}:L\prod _{K}K^{'} \to K^{'}$ is the pullback fibration over $K^{'}$ induced by $f$, then, $secat (p^{'})\leq secat (p)$.
\end{proposition} 

\begin{proof}
Let $K=\bigcup _i K_{i}$ be the union of its subcomplexes. Obviously, since $f$ is a simplicial map for any $i$, $f^{-1}(K_{i})$ is a subcomplex of $K^{'}$ and $K^{'}=\bigcup _{i}f^{-1}(K_{i})$. Now assume that for each $i$ there exists a simplicial map $s_{i}:K_i \to L$ satisfying $p\circ s_i =i _{K_i}$, where $i_{K_{i}}:K_i \to K$ is the inclusion map. We define ${s_i }^{'}:f^{-1}(K_i) \to L\prod _K K{'}  \subseteq L\prod K'$ by ${s_i}^{'}(v)=(s_i \circ f(v),v)$. Since  ${s_i}^{'}$ is a composition of simplicial maps, it is a simplicial map. Moreover,  $p{'} \circ {s_i}{'}(v)=p^{'}(s_i \circ f(v),v)=v$ by the definition of a pullback. Therefore $s_i'$ is a section for $p'$, and then $secat (p{'})\leq secat (p)$. 
\end{proof}

It is important to minimize the number of rules needed to define an algorithm for the movement of a robot. This can be accomplished by utilizing the relative topological complexity, as shown in \cite{rel}. The same principle is proven for targeted simplicial complexity, as demonstrated in the following corollary.

\begin{corollary}\label{c}
For $L\subseteq K$, $TC(K,L)\leq TC(K)$.
\end{corollary}

\begin{proof}
By Theorem
\ref{b}, $TC(K)=secat (p)$ and by Theorem \ref{i}, $TC(K,L)=secat (\pi )$. Also,  the map $\pi :P(K,L) \to K \prod L$ can be considered as the pullback of $p: PK \to K\prod K$ induced by the inclusion map $i:K \prod L \to K \prod K$. Then Proposition \ref{a} implies the inequality.
\end{proof}

\section{Invariance and relationship with the LS-category}\label{category}

In this section we study some homotopy properties of  targeted simplicial complexity. First we check the homotopy invariance of $TC(K, L)$ and then  we provide some close relationships between the LS-category and the targeted simplicial complexity.
Discrete topological complexity $TC$ is an invariant from the category of simplicial complexes, that only depends on the homotopy type \cite{dis}. Now we intend to check if the targeted case is strongly homotopy invariant. To verify we need to recall the definition of strong homotopy type.
 
\begin{definition}[\cite{dis}]
Let $K$ and $K^{'}$ be two finite simplicial complexes and let $L\subseteq K$ and $L^{'}\subseteq K^{'}$. We say that $(K,L)$ and $(K^{'},L^{'})$ have the same strong homotopy type if:
\begin{enumerate}
\item[(i)]
There exist simplicial maps $\phi :K\to K^{'}$ and $\psi :K^{'}\to K $ such that $\phi (L)\subseteq L^{'}$ and $\psi (L^{'})\subseteq L$.
\item[(ii)]
There is a p-homotopy $F:K\to PK$ such that $\alpha \circ F=id_K$, $\omega \circ F=\psi \circ \phi $ and $F(v)(t)\in L$ for all $v\in L$.
\item[(iii)]
There is a p-homotopy $H:K^{'}\to PK^{'}$ such that $\alpha \circ H =id_{K^{'}}$, $\omega \circ H=\phi \circ \psi $ and $H(w)(t)\in L^{'}$ for all $w\in L^{'}$.
\end{enumerate}

\end{definition}
By the following theorem we observe that $TC(K,L)$ depends only on the strong homotopy type of the pairs.

\begin{theorem}
Let $K$ and $K^{'}$ be two finite simplicial complexes, $L\subseteq K$ and $L^{'}\subseteq K^{'}$. If $(K,L)$ and $(K^{'},L^{'})$ have the same strong homotopy type, then $TC(K,L)=TC(K^{'},L^{'})$.

\begin{proof}
Let there exist $\phi :K\rightarrow K^{'}$ and $\psi :K^{'}\to K $ such that $\phi (L)\subseteq L^{'}$ and $\psi (L^{'})\subseteq L$, and let there be a p-homotopy $F:K\to PK$ such that $\alpha \circ F=id_K$, $\omega \circ F=\psi \circ \phi $ and $F(v)(t)\in L$ for all $v\in L$. We show that $TC(K,L)\leq TC(K^{'},L^{'})$. Suppose that $\Omega \subseteq K^{'}\prod L^{'}$ and there exists simplicial map $\sigma :\Omega \to P(K^{'},L^{'})$ such that $\pi \circ \sigma =i_{\Omega }$ where $\pi :P(K^{'},L^{'})\to  K^{'}\prod L^{'}$ is the finite-fibration and $i_{\Omega }$ is the inclusion map. Define $\Lambda =(\phi \prod \phi \vert _L)^{-1}(\Omega )\subseteq K \prod L$ and $\lambda :\Lambda \to P(K,L)$ by 
\[
\lambda (v,w)= F(v) \ast \psi \circ \sigma (\phi (v),\phi (w))\ast \overline{F(w)},
\]
where $\overline{F(w)}$ is the inverse of the Moore path $F(w)$ and $\ast$ is the product of Moore paths. This map is clearly a simplicial map and we have 
$\alpha \circ \lambda (v,w)=\alpha \circ F(v)=v$ and $\omega \circ \lambda (v,w)=\omega \circ \overline{F(w)}=\alpha \circ F(w)=w$. Now assume that $TC(K^{'},L^{'})=n$, and then there exists a covering $K^{'}\prod L^{'}=\Omega _0\cup ...\cup \Omega _n$ where $\Omega _j$, $j=0,...,n$, are subcomplexes satisfying in Definition \ref{def}. Then the corresponding $\Lambda _j=(\phi \prod \phi \vert _L)^{-1}(\Omega _j)\subseteq K \prod L$, $j=0,...,n$, make a covering of $K \prod L$ satisfying Definition \ref{def}. Hence $TC(K,L)\leq n$. The converse inequality can be proved by the same way.
\end{proof}
\end{theorem}

In \cite{dis} it was proven that $scat(K) \leq TC(K)$ and in \cite{rel}, the inequality $cat (X)\leq TC(X,Y)$ was shown. Now we aim to find a similar result for $TC(K,L)$.
From this point forward, assume that $K$ is a finite complex, $v_0 \in K$ a fixed vertex in $K$ and $p_0:P(K,v_0)\to K\prod \lbrace v_0 \rbrace $ is the path fibration where $P(K,v_0)=\lbrace      \gamma \in PK: \omega (\gamma )=v_0      \rbrace$ and $p_0 (\gamma)=(\alpha (\gamma),v_0)$. It is important to note that $p_0$ is the pullback of the fibration $p$, over the inclusion map $K\prod \lbrace v_0 \rbrace \hookrightarrow K\prod K$. Let  $\omega_1: P_0 K \to K$ be the terminal simplicial map where $P_0 K = \{ \gamma \in PK :\, \alpha_1 (\gamma )=v_0\}$.

\begin{lemma}\label{d}
Two simplicial maps $\omega_1$ and $p_0$ have the same \v{S}varc genus;
$secat(\omega_1)=secat(p_0)$.
Moreover, if $K$ is a connected finite complex, then $scat(K)=TC(K, v_0)$.
\end{lemma}

\begin{proof}
Suppose that $secat(\omega_1)=n$. Then there exist subcomplexes $\Lambda _0,...,\Lambda _n$ of $K$ and simplicial maps $\lambda _i :\Lambda _i\to P_0K$ such that $\omega_1 \circ \lambda _i=i_{\Lambda _i}$ for all $0\leq i\leq n$. We define $\Omega _i =\Lambda _i \prod \lbrace v_0 \rbrace $. It is clear that $\lbrace \Omega _i \rbrace _0 ^n$ covers $K\prod \lbrace  v_0\rbrace$. Put $\sigma _i:\Omega _i\to P(K,v_0)$ by $\sigma _i (v,v_0)=\overline{\lambda _i(v)}$. We have $p_0 \circ \sigma _i(v,v_0)=(\alpha (\sigma _i(v,v_0)),\omega(\sigma _i(v,v_0)))=(\alpha \overline{(\lambda _i(v))},\omega (\overline{\lambda _i(v))})=(v,v_0)$.
Then $secat(p_0)\leq secat(\omega)$. Now suppose that $secat(p_0)=m$. Thus there exist subcomplexes $\Omega _0,...,\Omega _m$ of $K\prod \lbrace v_0 \rbrace$ and simplicial maps $\sigma _i:\Omega _i\to P(K,v_0)$ such that $p_0 \circ \sigma _i=i_{\Omega _i}$ for all $0\leq i\leq m$. We define $\Lambda _i$ as the projection of $\Omega _i$ on the first component and $\lambda _i:\Lambda _i\to P_0K$ by $\lambda _i(v)=\overline{\sigma _i(v,v_0)}$. Hence we have $\omega \circ \lambda _i(v)=\omega \overline{(\sigma _i (v,v_0))}=v$ and  then $secat(\omega)\leq secat(p_0)$. Therefore $secat(\omega)= secat(p_o)$.

Now Let $K$ be a connected finite complex. Since $K$ is finite, by Theorem \ref{i}, $TC(K,v_0)=secat (p_0)$.
Also since $K$ is connected and finite, by Theorem \ref{th2.13n}, $scat (K) = secat (\omega_1)$. Thus $TC(K,v _0)=scat(K)$. 
\end{proof}


Lemma \ref{d} shows that targeted simplicial complexity, where the target is a point, equals the LS-category of $K$. 
Now we show that the LS-category of $K$ is a lower bound for the targeted simplicial complexity. 
In \cite{rel}, Short compared $TC(X,Y)$ with the LS-categories $cat(X)$ and $cat(X\times X)$ and he obtained the inequalities $cat(X)\le TC(X,Y) \le cat(X \times X)$. We prove a  discrete version in the following.

\begin{proposition}\label{ine}
Let $K$ be a connected finite complex and $L\subseteq K$ be a non-empty subcomplex of $K$. Then 
\[
scat(K)\leq TC(K,L)\leq scat (K\prod K).
\]
\end{proposition}

\begin{proof}
By \cite[Theorem 4.4]{dis}, we have $TC(K)\leq scat (K\prod K)$. Also Corollary \ref{c} shows that $TC(K,L)\leq TC(K)$. Therefore, $TC(K,L)\leq scat (K\prod K)$. Now we prove that $scat(K)\leq TC(K,L)$. By Lemma \ref{d}, we have $scat(K)=secat(p_0)$. It is enough to show that $secat(p_0)\leq TC(K,L)$. Assume that $TC(K,L)=n$. Then there exist subcomplexes $\Omega _0,...,\Omega_n$ of $K \prod L$ and simplicial maps $\sigma _i:\Omega _i\to P(K,L)$ such that $\pi \circ \sigma _i=i_{\Omega _i}$ for all $0\leq i\leq n$. Define $\Lambda _i=\Omega _i \cap (K\prod \lbrace v_0\rbrace)\subseteq K\prod \lbrace v_0\rbrace$ and $\lambda _i=\sigma _i \vert _{\Lambda _i}:\Lambda _i\to P(K,v_0)$. We have $p_0 \circ \lambda _i(v,v_0)=(\alpha(\lambda_i(v,v_0)),\omega (\lambda_i(v,v_0)))= \pi \circ \lambda _i (v,v_0)=(v,v_0)$. Hence $secat(p_0)\leq TC(K,L)$ and then $scat (K) \le TC(K,L)$.
\end{proof}

In \cite{dis}, it was shown that $K$ is strongly collapsible if and only if $TC(K)=0$. We can conclude a similar result for $TC(K,L)$. Note that $K$ is strongly collapsible if and only if $scat (K) =0$, because $scat$ counts the number of categorical subcomplexes; See Definition \ref{de2.12n}.

\begin{proposition}
$TC(K,L)=0$ if and only if $K$ is strongly collapsible.
\end{proposition}

\begin{proof}
Suppose that $TC(K,L)=0$. Since $ scat (K)\leq TC(K,L)$, we have $ scat (K) = 0$. Therefore by Definition \ref{de2.12n}, $K$ is strongly collapsible. Now suppose that $K$ is strongly collapsible. Then by \cite[Corollary 4.5]{dis}, $TC(K)=0$. Thus by Corollary \ref{c}, $TC(K,L)\leq TC(K)=0$ and therefore $TC(K,L)=0$.
\end{proof}

In Lemma \ref {d} we see that if the target is a point, then $TC(K, v_0) = scat (K)$. This result can be generalized to categorical targets as follows. 

\begin{theorem}\label{intr}
Let $K$ be a finite connected complex and $L$ be a categorical subcomplex of $K$. Then $TC(K,L)=scat(K)$.
\end{theorem}

\begin{proof}
We see in Proposition \ref{ine} that $scat(K)\leq TC(K,L)$. It  remains to prove that $TC(K,L)\leq scat(K)$. Assume that $scat(K)=n$. Then by Lemma \ref{d}, we have $secat(p_0)=n$. Let  $\Omega _0,...,\Omega_n$ be such subcomplexes of $K\prod \lbrace v_0\rbrace$ with sections $\sigma _i$ over each $\Omega _i$. Let $\pi _1(\Omega _i)$ be the projection of $\Omega _i$, onto its $K$ component. Define $\Lambda _i=\pi _1(\Omega _i)\prod L$. Since $\lbrace  \pi _1(\Omega _i)\rbrace _0 ^n$ covers $K$, the collection of $\Lambda _i$ covers $K \prod L$. Now suppose that $H:L\prod [0,m]\to K$ is a p-homotopy such that $H(-,0)=c_{v_0}$ and $H(-,m)=i_L$. Put $h_l:=H( l,-)$ for all $l\in L$. One can see that $h_l$ is a Moore path from $v_0$ to $l$. Define $\lambda _i:\Lambda _i\to P(K,L)$ by $\lambda _i(v,l)=\sigma _i(v,v_0)\ast h_l$, where $\ast $ is the product of Moore paths. Clearly $\lambda _i$ is a simplicial map and we have $\pi \circ \lambda _i (v,l)=(\alpha(\lambda _i(v,l)),\omega(\lambda _i(v,l)))=(\alpha(\sigma _i(v,v_0)),\omega(h_l))=(v,l)$. Thus $TC(K,L)\leq scat(K)$.
\end{proof}
  
\section{Geometric Realization}
In this section, we compare the targeted simplicial complexity \( TC(K, L) \) with the usual relative topological complexity \( TC(|K|, |L|) \), where \( |\cdot| \) denotes the geometric realization functor. In \cite{Inv}, it was discussed that topological complexity can be computed by considering closed subspaces instead of open subspaces. This fact can be applied for geometric realization of any finite simplicial complex. We intend to prove the following theorem by using this statement.

\begin{theorem}\label{e}
$TC(|K|,|L|)\leq TC(K,L)$.
\end{theorem}

\begin{proof}
Let $TC(K,L)=n$. Then there exist subcomplexes $\Omega_0,...,\Omega_n$ of $K \prod L$ and simplicial maps $\sigma_i:\Omega_i\to P(K,L)$ for $0\leq i\leq n$ such that $\pi \circ \sigma_i=i_{\Omega_i}$.
We know that 
there exist some homeomorphism $|K\times L| \to |K|\times |L|$ and  homotopy equivalence 
$v: |K \prod L| \to |K| \times |L|$. Consider the simplicial map $g:K\times L\to K \prod L$ where $g(s_1\times s_2)=s_1\sqcap s_2$ (see \cite{com} and \cite{Hat}).  Note that $|g|$ can be assumed as the homotopy inverse of $v$.
Put $u=|g|$. Define $U_i=u^{-1}(|\Omega_i|) \subseteq |K|\times|L|$ and $s_i:U_i\to P(|K|,|L|)$ by $s_i(x_1,x_2)=f\circ |\sigma_i|\circ u(x_1,x_2)$  where 
$f:|P(K,L)| \to P(| K|,| L|)$ is  lifting of the map $ v \circ |\pi| : |P(K,L)| \to |K|\times |L|$. Since the natural map $\pi': P(|K|,|L|) \to |K| \times |L|$ is a fibration, $f$ is continuous. 
Moreover, $\pi' \circ s_i = \pi' \circ f\circ |\sigma_i|\circ u|_{U_i} = v \circ |\pi| \circ |\sigma_i|\circ u|_{U_i}$, because $f$ is the lifting of $ v \circ |\pi|$ by $\pi'$. Also, since $|\cdot|$ is a functor $|\pi| \circ |\sigma_i| = |\pi \circ \sigma_i| = | i_{\Omega_i}|$. 
Therefore
\[
\pi' \circ s_i= v \circ |i_{\Omega_i}| \circ u|_{U_i}  = v \circ |i_{\Omega_i} \circ g|_{g^{-1} (\Omega_i)}|  = v \circ |g|_{g^{-1} (\Omega_i)}| = v \circ u|_{U_i} \simeq id|_{U_i} = i_{U_i},
\]
where $i_{U_i}$ is the inclusion map. Hence 
$TC(|K|,|L|)\leq TC(K,L)$.
\end{proof}

Now, we need to provide some useful notes using reference \cite{dis} to present an example for the computation of \( TC(K, L) \) and subsequently prove a theorem at the end.
Remark \ref{f} recalls applicable statements about subcomlexes satisfying Definition \ref{def} to calculate $TC$ of complexes.

\begin{remark}[\cite{dis}]\label{f}
\begin{enumerate}
\item
Any categorical product of two strongly collapsible complexes is also strongly collapsible.

\item
Every subcomplex of a strongly collapsible complex is categorical and every subcomplex which is strongly collapsible itself, is categorical too.

\item
Any categorical subcomplex of $K \prod L$ satisfies in conditions of Definition \ref{def} of $TC(K, L)$.
\end{enumerate}
\end{remark}

Now we are  able to study examples of complexes and present targeted simplicial complexity  of some these complexes. To do this, we use topological complexity of geometric realization of simplicial complexes which were previously studied and calculated in \cite{rel}.
Moreover, Theorem \ref{e} implies that the targeted simplicial complexity is bounded from below by the topological complexity of geometric realization.

\begin{example}\label{exa}
Let $K$ be the wedge of two triangulated circles, one of which is triangulated by four edges and the other one is triangulated by three edges. Put $K:=K_{2,2}\vee \partial \Delta ^{2}$ where complex $K_{2,2}$ is a bipartite gragh; That is $K$ is equal to
\[
\lbrace \emptyset ,\lbrace a_1\rbrace ,\lbrace a_2 \rbrace ,\lbrace b_1\rbrace ,\lbrace b_2\rbrace ,\lbrace a\rbrace ,\lbrace c\rbrace ,\lbrace a_1,b_1\rbrace ,\lbrace a_1,b_2\rbrace ,\lbrace a_2,b_1\rbrace ,\lbrace a_2,b_2\rbrace ,\lbrace a,c\rbrace ,\lbrace a,b_2\rbrace ,\lbrace b_2,c\rbrace \rbrace ,
\]

\noindent
and $L=\partial \Delta ^{2} =\lbrace \emptyset ,\lbrace a\rbrace ,\lbrace b_2 \rbrace ,\lbrace c\rbrace ,\lbrace b_2,c\rbrace ,\lbrace a,c\rbrace ,\lbrace a,b_2\rbrace\rbrace$ whose geometric realizations are shown in (a) and (b) of Figure \ref{Fig1} respectively.

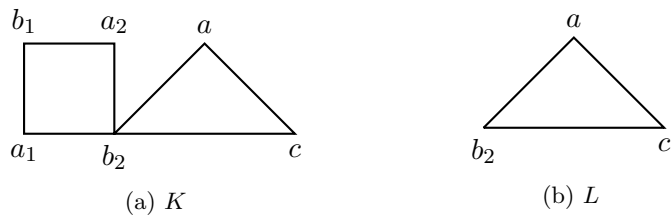
\begin{figure}[!ht]
\centering
\parbox{6cm}{
\subfloat[$K$]
{\begin{tikzpicture}[scale=0.6,  thick,simple/.style={ball color=black,inner sep=.8mm,circle,color=black,text=black, minimum size=0mm},‌
Ledge/.style={to path={
.. controls +(0:1) and +(45:1) .. (\tikztotarget) \tikztonodes}}
]
\node[above]  (a) at ( 2,2){$a$};
\node[below] (b_2) at ( 0,0){$b_2$};‌
\node[below] (c) at ( 4,0){$c$};‌‌‌
\node[above](a_2) at (0,2){$a_2$};
\node[above](a_1) at (-2,2){$b_1$};
\node[below](b_1) at (-2,0){$a_1$};
\draw(0,0)--(2,2)--(4,0)--(0,0)--(0,2)--(-2,2)--(-2,0)--(0,0);
\end{tikzpicture}}}
%
\parbox{2.5cm}{
\subfloat[$L$]
{\begin{tikzpicture}[scale=0.6,  thick,simple/.style={ball color=black,inner sep=1cm,circle,color=black,text=black, minimum size=0mm},‌
Ledge/.style={to path={
.. controls +(0:1) and +(45:1) .. (\tikztotarget) \tikztonodes}}
]
\node[above]  (a) at ( 2,2){$a$};
‌\node[below] (b_2) at ( 0,0){$b_2$};‌
\node[below] (c) at ( 4,0){$c$};‌‌‌
\draw(0,0)--(2,2)--(4,0)--(0,0);
\end{tikzpicture}}}
\caption{Geometric realization of $K$ and $L$.}\label{Fig1}
\end{figure}


Since $|K|$ is homeomorphic to $S^1 \vee S^1$ and $|L|$ is homeomorphic to  $S^1$, by \cite[Proposition 3.13]{rel} we have $TC(|K|,|L|)=2$. By Theorem \ref{e}, we can conclude that $TC(K,L)\geq 2$. It is enough to exhibit three subcomplexes covering $K \prod L$. Figure \ref{Fig 3} represents the categorical product $K \prod L$.

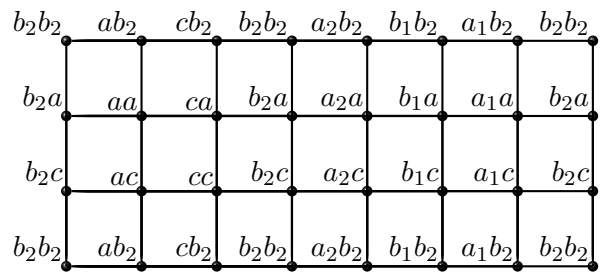
\begin{figure}[!ht]
\centering
\begin{tikzpicture}[  thick,simple/.style={ball color=black,inner sep=.5mm,circle,color=black,text=white, minimum size=0mm},‌
Ledge/.style={to path={
.. controls +(0:1) and +(45:1) .. (\tikztotarget) \tikztonodes}}
]
\foreach \x/\xtext in  {1/b_2,2/a,3/c,4/b_2,5/a_2,6/b_1,7/a_1,8/b_2}
\foreach \y/\ytext in {1/b_2,2/c,3/a,4/b_2}
{
\node[label={[label distance=-.2cm]145:‌${\tiny\xtext\ytext}$}] (\xtext\ytext) at (1*\x cm,1*\y cm)[simple]{};
}
\foreach \x/\xtext in  {1/b_2,2/a,3/c,4/b_2,5/a_2,6/b_1,7/a_1,8/b_2}
\foreach \y/\ytext in {1/b_2,2/c,3/a,4/b_2}
{
\draw(1*\x cm,1*\y cm) -- (1*\x+1 cm,1*\y cm)[simple]{};
\draw(1*\x cm,1*\y cm) -- (1*\x cm,1*\y+1 cm)[simple]{};
}
\end{tikzpicture}
\caption{Representation of simplicial complex $K \prod L$.}\label{Fig 3}
\end{figure}


Figure \ref{Fig 4} contains  two subcomlexes of $K \prod L$  satisfying Definition \ref{def}. Both of  $\Omega _1 $ and $\Omega _2 $ are the product of two strongly collapsible complexes, and then they are strongly collapsible. Let  $\Omega _3$ be the complement of $\Omega_1 \cup \Omega_2$ in $K \prod L$. It is subcomplex of a strongly collapsible complex which is the product of two strongly collapsible complexes. Then by Remark \ref{f}, they are strongly collapsible and whence they are subcomplexes satisfying Definition \ref{def}. Hence $TC(K,L)\leq 2$ and then $TC(K,L)=2$.
 
 \begin{figure}[!ht]
\centering
\parbox{6.5cm}{
\subfloat[$\Omega_1$]
{\begin{tikzpicture}[scale = 0.8, thick,simple/.style={ball color=black,inner sep=.5mm,circle,color=black,text=white, minimum size=0mm},‌
Ledge/.style={to path={
.. controls +(0:1) and +(45:1) .. (\tikztotarget) \tikztonodes}}
]
\filldraw[gray!30] (1cm,2cm) -- (3cm,2cm)--(3cm,4cm) --(1cm,4cm);
\filldraw[gray!30] (4cm,2cm) -- (7cm,2cm)--(7cm,4cm) --(4cm,4cm);
\foreach \x/\xtext in  {1/b_2,2/a,3/c,4/b_2,5/a_2,6/b_1,7/a_1,8/b_2}
\foreach \y/\ytext in {1/b_2,2/c,3/a,4/b_2}
{
\draw(1*\x cm,1*\y cm) -- (1*\x+1 cm,1*\y cm)[simple]{};
\draw(1*\x cm,1*\y cm) -- (1*\x cm,1*\y+1 cm)[simple]{};
}
\end{tikzpicture}}}
\parbox{5cm}{
\subfloat[$\Omega_2$]
{\begin{tikzpicture}[scale = 0.8, thick,simple/.style={ball color=black,inner sep=.5mm,circle,color=black,text=white, minimum size=0mm},‌
Ledge/.style={to path={
.. controls +(0:1) and +(45:1) .. (\tikztotarget) \tikztonodes}}
]
\filldraw[red!30] (2cm,1cm) -- (4cm,1cm)--(4cm,3cm) --(2cm,3cm);
\filldraw[red!30] (5cm,1cm) -- (8cm,1cm)--(8cm,3cm) --(5cm,3cm);
\foreach \x/\xtext in  {1/b_2,2/a,3/c,4/b_2,5/a_2,6/b_1,7/a_1,8/b_2}
\foreach \y/\ytext in {1/b_2,2/c,3/a,4/b_2}
{
\draw(1*\x cm,1*\y cm) -- (1*\x+1 cm,1*\y cm)[simple]{};
\draw(1*\x cm,1*\y cm) -- (1*\x cm,1*\y+1 cm)[simple]{};
}
\end{tikzpicture}}}
\caption{Subcomplexes of $K \prod L$ satisfying Definition \ref{def}. }\label{Fig 4}
\end{figure}
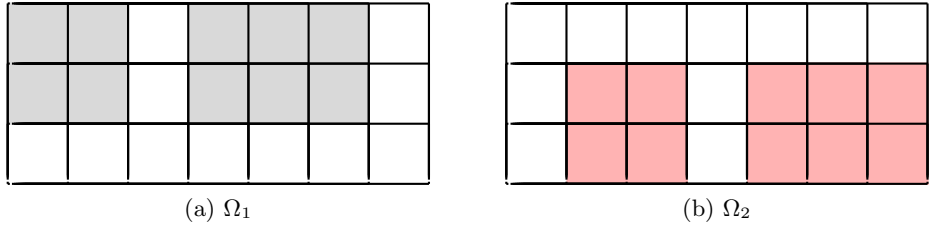
\end{example}
  
Example \ref{exa} can be generalized  to any complex obtained by the wedge of arbitrary number of circles as proven in Theorem \ref{th5.4n}. The sketch of proof is similar, but more details are needed to check.

\begin{theorem}\label{th5.4n}
Let  $m\leq n$ and also let  $K$ and $L\subseteq K$ be the wedge of $n$ and $m$ triangulated circles respectively. Then $TC(K,L)=2$.
\end{theorem}

\begin{proof}
Note that $n$ and $m$ are the first Betti number of $|K|$ and $|L|$.
By \cite [Proposition 3.13]{rel}, we have $TC(|K|,|L|)=2$. Then by Theorem \ref{d}, $TC(K,L)\geq 2$. Suppose that $K$ is the wedge of $n_1$ circles  triangulated with three edges, $n_2$ circles triangulated with four edges, ... and $n_l$ circles  triangulated with $l+2$ edges. Also let $L$ be the wedge of $m_1$ circles  triangulated with three edges, $m_2$ circles  triangulated with four edges, ... and $m_l$ circles  triangulated with $l+2$ edges. Then $m_1\leq n_1$, $m_2\leq n_2$, ... and $m_l\leq n_l$. 
We need three subcomplexes of $K \prod L$ satisfying Definition \ref{def}.
For better understanding, consider $K$ and $L$ as complexes represented in Figure \ref{Fig 5-1}.
 
 
\begin{figure}[!ht]
\centering
\parbox{6cm}{
\subfloat[$K$]
{\begin{tikzpicture}[  thick,simple/.style={ball color=black,inner sep=.8mm,circle,color=black,text=black, minimum size=0mm},‌
Ledge/.style={to path={
.. controls +(0:1) and +(45:1) .. (\tikztotarget) \tikztonodes}}
]
\node[above left]  (a) at ( 0,0){$a$};
\node [left] (b_2) at (-1,.5){$b_2$};‌
\node[left](c_2) at (-1,-.5){$c_2$};
\node [below](b_1) at (-.5,-1){$b_1$};‌‌‌
\node[below](c_1) at (.5,-1){$c_1$};
\node[right](b_3) at (1,-.5){$b_3$};
\node[right](c_3) at (1,.5){$c_3$};
\node[right](b) at (.5,1){$b$};
\node[left](c) at (-.5,1){$c$};
\node[above](d) at (0,2){$d$};
\draw (0,0)--(-1,.5)--(-1,-.5)--(0,0)--(-.5,-1)--(.5,-1)--(0,0)--(1,-.5)--(1,.5)--(0,0)--(.5,1)--(0,2)--(-.5,1)--(0,0);
\draw[very thick] node[xshift=0 cm, yshift=3 cm]{}
;
\end{tikzpicture}}}
\parbox{4cm}{
\subfloat[$L$]
{\begin{tikzpicture}[  thick,vertex/.style={ball color=black,inner sep=.3mm,circle,color=black,text=black, minimum size=.5mm},‌
Ledge/.style={to path={
.. controls +(0:1) and +(45:1) .. (\tikztotarget) \tikztonodes}}
]
\node[left]  (a) at ( 0,0){$a$};
\node[above](d) at (0,2){$d$};
\node[below](c_1) at (.5,-1){$c_1$};
\node[right](b) at (.5,1){$b$};
\node[left](c) at (-.5,1){$c$};
\node [below](b_1) at (-.5,-1){$b_1$};‌‌‌
\draw (0,0)--(-.5,-1)--(.5,-1)--(0,0)--(.5,1)--(0,2)--(-.5,1)--(0,0);
\draw[very thick] node[xshift=0 cm, yshift=3 cm]{}
;
\end{tikzpicture}}}
\caption{A hypothetical example for better understanding.}\label{Fig 5-1}
\end{figure}
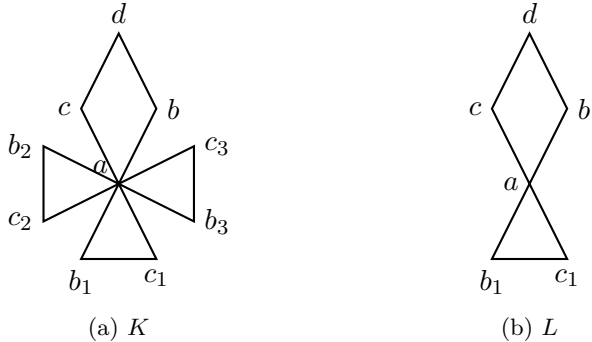

\noindent
Then the categorical product of $K$ and $L$ can be shown as in Figure \ref{Fig 5}.

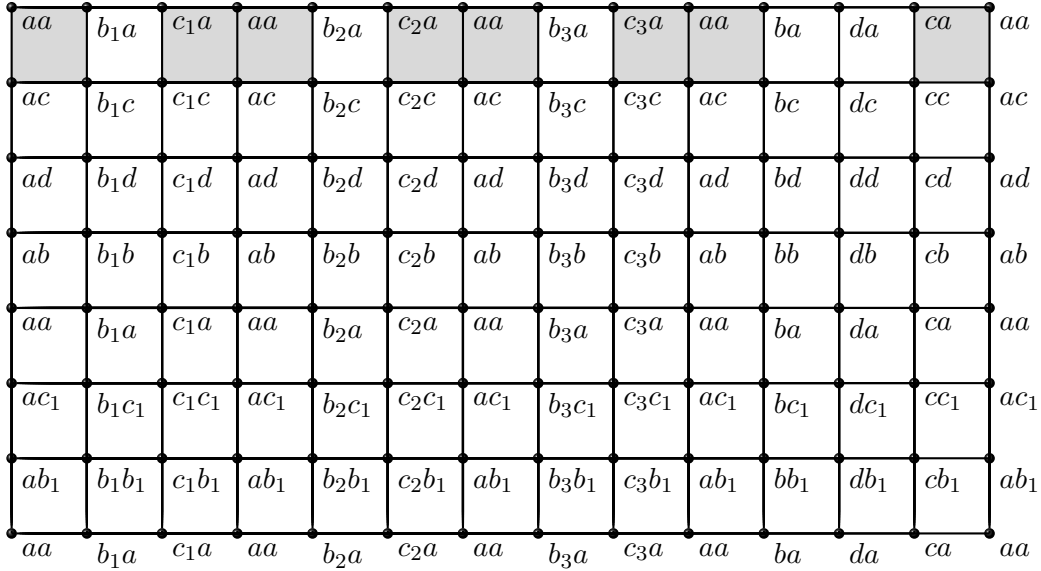
\begin{figure}[!ht]
\centering
\begin{tikzpicture}[ thick,simple/.style={ball color=black,inner sep=.5mm,circle,color=black,text=white, minimum size=0mm},‌
Ledge/.style={to path={
.. controls +(0:1) and +(45:1) .. (\tikztotarget) \tikztonodes}}
]
\filldraw[gray!30](1,7)--(2,7)--(2,8)--(1,8);
\filldraw[gray!30](3,7)--(4,7)--(4,8)--(3,8);
\filldraw[gray!30](4,7)--(5,7)--(5,8)--(4,8);
\filldraw[gray!30](6,7)--(7,7)--(7,8)--(6,8);
\filldraw[gray!30](7,7)--(8,7)--(8,8)--(7,8);
\filldraw[gray!30](9,7)--(10,7)--(10,8)--(9,8);
\filldraw[gray!30](10,7)--(11,7)--(11,8)--(10,8);
\filldraw[gray!30](13,7)--(14,7)--(14,8)--(13,8);

\foreach \x/\xtext in  {1/a,2/b_1,3/c_1,4/a,5/b_2,6/c_2,7/a,8/b_3,9/c_3,10/a,11/b,12/d,13/c,14/a}
\foreach \y/\ytext in {1/a,2/b_1,3/c_1,4/a,5/b,6/d,7/c,8/a}
{
\node[label={[label distance=-.1cm,scale=1]300:‌${\xtext\ytext}$}] (\xtext\ytext) at (1*\x cm,1*\y cm)[simple]{};
}
\foreach \x/\xtext in  {1/a,2/b_1,3/c_1,4/a,5/b_2,6/c_2,7/a,8/b_3,9/c_3,10/a,11/b,12/d,13/c,14/a}
\foreach \y/\ytext in {1/a,2/b_1,3/c_1,4/a,5/b,6/d,7/c,8/a}
{
\draw(1*\x cm,1*\y cm) -- (1*\x+1 cm,1*\y cm)[simple]{};
\draw(1*\x cm,1*\y cm) -- (1*\x cm,1*\y+1 cm)[simple]{};

}

\end{tikzpicture}
\caption{Representation of $K \prod L$ }\label{Fig 5}
\end{figure}

Let $H$ be the subcomplex of $K \prod L$ consisting of all maximal simplices in the same horizontal line.  
The complex  $H$ contains $2(n_1+n_2+...+n_l)$ tetrahedrons with a common edge, say $\tau _1, \tau _2,..., \tau_{2(n_1+...+n_l)}$; See Figure \ref{Fig 5}. Additionally, a strongly collapsible subcomplex of $H$ can not contain more than two tetrahedrons $\tau _i$, and if it contains two of them, these two tetrahedrons can not be in the same cycle. If not, then the  realization of  complex $H$ would contain a cycle or more.  Hence it is not contractible and then by Theorem \ref{j}, this subcomplex is not strongly collapsible.

Let $\Omega$ be a subcomplex of $K \prod L$ satisfying Definition \ref{def}, and let $i_0:K\to K \prod L$ be a map defined by $i_0 (v)=(v,v_0)$ where $v_0$ is an arbitrary vertex in $L$. We show that the subcomplex 
\(
\Lambda= i_0 ^{-1}(\Omega )=p _1 (K\times \lbrace v_0\rbrace \cap \Omega )\subseteq K
\)
is categorical in $K$ where $p_1$ is the projection map on the first component. By Definition \ref{def}, there exists map $\sigma :\Omega \to L$ such that $\Delta \circ \sigma \sim i_{\Omega}$ or $(\sigma ,\sigma )\sim i_{\Omega}$. Thus 
$(\sigma|_{\Lambda}, \sigma|_{\Lambda})\sim i_0|_{\Lambda}$ which implies that $id|_{\Lambda} \sim \sigma|_{\Lambda} \sim c_{v_0}|_{\Lambda}$. Therefore $id|_{i_0^{-1}(\Omega)}\sim c_{v_0}$ which means that ${i_0^{-1}(\Omega)}$ is categorical.
So $i_0 ^{-1}(\Omega)$ is not equal to $K$ and it can not contain any cycle of $K$. Then $\Omega \cap H$ contains at most $2n_1+3n_2+(l+1)n_l$ tetrahedrons. Also none of other $n_1+n_2+...+n_l$ tetrahedrons have common vertical edge, because if not,  $i_0^{-1}(\Omega)$ would contain a cycle, which is a contradiction.

Analogously, let $V$ be the subcomplex consisting of all the maximal simplices of $K \prod L$ in the same vertical line. We observe that $\Omega \cap V$ contains at most $2m_1+3m_2+(l+1)m_l$ tetrahedrons and no more of $m_1+m_2+...+m_l$ have a common horizontal edge. We find a covering of $K \prod L$ by three subcomplexes satisfying Definition \ref{def}; $K \prod L=\Omega _1 \cup \Omega _2 \cup \Omega _3$, which $\Omega _1$ and $\Omega _2$ are exhibited in Figure \ref{Fig 6}, and $\Omega _3$ is the complement of $\Omega _1 \cup \Omega _2$ in $K \prod L$. One can verify that these subcomplexes are categorical by using a similar argument as used in Example \ref{exa}.

\begin{figure}[!ht]
\centering
\begin{tikzpicture}[scale = 0.6,  thick,simple/.style={ball color=black,inner sep=.5mm,circle,color=black,text=white, minimum size=0mm},‌
Ledge/.style={to path={
.. controls +(0:1) and +(45:1) .. (\tikztotarget) \tikztonodes}}
]
\filldraw[gray!30] (1cm,6cm) -- (3cm,6cm)--(3cm,8cm) --(1cm,8cm);
\filldraw[gray!30](4,6)--(6,6)--(6,8)--(4,8);
\filldraw[gray!30](7,6)--(9,6)--(9,8)--(7,8);
\filldraw[gray!30](10,6)--(13,6)--(13,8)--(10,8);
\filldraw[gray!30] (1cm,2cm) -- (3cm,2cm)--(3cm,5cm) --(1cm,5cm);
\filldraw[gray!30](4,2)--(6,2)--(6,5)--(4,5);
\filldraw[gray!30](7,2)--(9,2)--(9,5)--(7,5);
\filldraw[gray!30](10,2)--(13,2)--(13,5)--(10,5);

\foreach \x/\xtext in  {1/a,2/b_1,3/c_1,4/a,5/b_2,6/c_2,7/a,8/b_3,9/c_3,10/a,11/b,12/d,13/c,14/a}
\foreach \y/\ytext in {1/a,2/b_1,3/c_1,4/a,5/b,6/d,7/c,8/a}
{
\draw(1*\x cm,1*\y cm) -- (1*\x+1 cm,1*\y cm)[simple]{};
\draw(1*\x cm,1*\y cm) -- (1*\x cm,1*\y+1 cm)[simple]{};

}
\end{tikzpicture}
\caption{The subcomplex $\Omega_1$}
\end{figure}
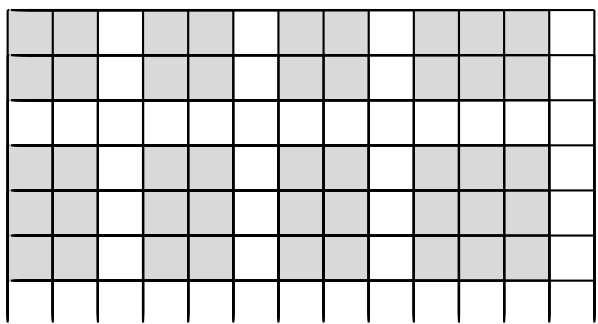

\begin{figure}[!ht]
\centering
\begin{tikzpicture}[scale = 0.6,  thick,simple/.style={ball color=black,inner sep=.5mm,circle,color=black,text=white, minimum size=0mm},‌
Ledge/.style={to path={
.. controls +(0:1) and +(45:1) .. (\tikztotarget) \tikztonodes}}
]
\filldraw[red!30] (2cm,5cm) -- (4cm,5cm)--(4cm,7cm) --(2cm,7cm);
\filldraw[red!30](5,5)--(7,5)--(7,7)--(5,7);
\filldraw[red!30](8,5)--(10,5)--(10,7)--(8,7);
\filldraw[red!30](11,5)--(14,5)--(14,7)--(11,7);
\filldraw[red!30] (2cm,1cm) -- (4cm,1cm)--(4cm,4cm) --(2cm,4cm);
\filldraw[red!30](5,1)--(7,1)--(7,4)--(5,4);
\filldraw[red!30](8,1)--(10,1)--(10,4)--(8,4);
\filldraw[red!30](11,1)--(14,1)--(14,4)--(11,4);

\foreach \x/\xtext in  {1/a,2/b_1,3/c_1,4/a,5/b_2,6/c_2,7/a,8/b_3,9/c_3,10/a,11/b,12/d,13/c,14/a}
\foreach \y/\ytext in {1/a,2/b_1,3/c_1,4/a,5/b,6/d,7/c,8/a}
{
\draw(1*\x cm,1*\y cm) -- (1*\x+1 cm,1*\y cm)[simple]{};
\draw(1*\x cm,1*\y cm) -- (1*\x cm,1*\y+1 cm)[simple]{};
}
\end{tikzpicture}
\caption{The subcomplex $\Omega_2$}\label{Fig 6}
\end{figure}
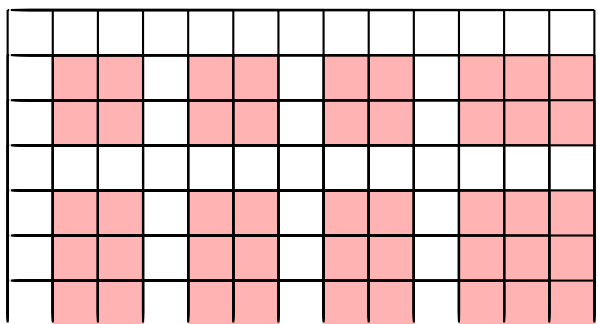
\end{proof}

\section{$m$-step simplicial complexity}

In this section, we intend to introduce a kind of simplicial complexity to count the motion planners connecting vertices by $m$-paths; Paths with length smaller than or equal to a fixed integer $m \ge 0$. This complexity helps us to make algorithms for robots whose motions are limited; For instance those robots which needs to be charged after traversing $m$ steps. 

Let $K$ be a simplicial complex and let $\ell (\gamma)$ denote the length of Moore path $\gamma$; That is the number of vertices between $\alpha_1 (\gamma)$ and $\omega_1(\gamma)$. We define $P_m(K)=\lbrace \gamma \in P(K) : \ell (\gamma) \le m\} $ and $\pi :P_m(K)\to K\prod K$ which maps any path to the start and end vertices of the path.

\begin{definition}
Let $K$ be a simplicial complex. Define the $m$-step simplicial complexity $TC_m(K)$ of complex $K$ as the \v{S}varc genus of simplicial map $\pi: P_m(K)\to K\prod K$; $TC_m(K) = secat (\pi_m)$.
\end{definition}

We know that $TC_{\infty}(K) = TC (K)$ for any simplicial complex and $TC_m$ is a descending sequence converging to $TC$. 
The $m$-step simplicial complexity depends on the length of paths between vertices; As an example, for simplicial complex $K$ with complete $1$-skeleton, $TC_m (k)= TC (k)$ for each $m \in \mathbb{N}$.

Let $K$ be a simplicial complex and $L\subseteq K$ be a subcomplex. 
We define $P_m(K,L)=\lbrace \gamma \in P(K,L) : \ell (\gamma) \le m\} $ and $\pi_m :P_m(K,L)\to K\prod  L$ as the restriction
of $\pi$ to $P_m(K,L)$.

\begin{definition}
Let $K$ be a simplicial complex and $L\subseteq K$ be a subcomplex. Define the $m$-step targeted complexity $TC_m(K,L)$ of pair $(K,L)$ as the \v{S}varc genus of simplicial map $\pi_m$; $TC_m(K,L) = secat (\pi_m)$.
\end{definition}

In other words $TC_m(K,L)$ is the number of simplicial sections of the simplicial map $\pi_m :P_m(K,L) \to K \prod L$.


\begin{proposition}
Let $K$ be a simplicial complex and $L \subseteq K$ be a subcomplex of $K$. Then 
\[
TC_m(K, L) \le TC_m (K).
\]
\end{proposition}

\begin{proof}
Let $TC_m (K) = n$. Then there exist subcomplexes $\Omega'_0, \ldots, \Omega'_n$ and sections $\sigma'_i: \Omega'_i \to P_m(K)$ for $\pi_m$ such that $K \prod K = \bigcup_0^n \Omega_i$. Define $\Omega_i = \Omega'_i \cap \big( K \prod L\big)$ and $\sigma_i = \sigma'_i |_{\Omega_i}$. Since $\sigma'_i$ is a section for $\pi_m$,  so is $\sigma_i$.
\end{proof}

By (relative) $m$-homotopy between $f, f': K \to K'$, we mean a simplicial map $F: K \to P_m K (P_m (K, L))$ such that $\alpha \circ F = f$ and $\omega \circ F = f'$  and we denote it  by $F: f \simeq f'$.

\begin{definition}
Let $K$ and $K^{'}$ be two finite simplicial complexes and let $L\subseteq K$ and $L^{'}\subseteq K^{'}$. We say that $(K,L)$ and $(K^{'},L^{'})$ have the same $m$-strong homotopy type if:
\begin{enumerate}
\item[(i)]
There exist simplicial maps $\phi :K\to K^{'}$ and $\psi :K^{'}\to K $ such that $\phi (L)\subseteq L^{'}$ and $\psi (L^{'})\subseteq L$.
\item[(ii)]
There is a $m$-homotopy $F:K\to P_m K$ such that $\alpha \circ F=id_K$, $\omega \circ F=\psi \circ \phi $ and $F(v)(t)\in L$ for all $v\in L$.
\item[(iii)]
There is a $m$-homotopy $H:K^{'}\to P_m K^{'}$ such that $\alpha \circ H =id_{K^{'}}$, $\omega \circ H=\phi \circ \psi $ and $H(w)(t)\in L^{'}$ for all $w\in L^{'}$. 

If $L = L' = \emptyset$, then $K$ and $K'$ have the same $m$-strong homotopy type.
\end{enumerate}

\end{definition}
By the following theorem we observe that $TC_m(K,L)$ depends only on the $m$-strong homotopy type of the pairs.

\begin{theorem}
Let $K$ and $K^{'}$ be two finite simplicial complexes, $L\subseteq K$ and $L^{'}\subseteq K^{'}$. If $(K,L)$ and $(K^{'},L^{'})$ have the same strong $m$-homotopy type, then $TC_m (K,L)=TC_m (K^{'},L^{'})$.
\end{theorem}

The proof is similar to proof of Theorem  4.2 by replacing $P(K',L')$ with $P_m(K',L')$.

\begin{definition}
Let $K$ be a complex. The $m$-simplicial category of $K$, denoted by $scat_m(K)$ is defined as the least integer $n$ such that $K$ is covered by $n +1$ nember of $m$-categorical subcomplexes $\Omega_0, \ldots, \Omega_n$; That is there exists $m$-homotopy $F_i : \iota_{\Omega_i} \simeq c_v$ for some $v \in K$ and each $i =0, \ldots, n$, where $\iota$ is the inclusion map.
\end{definition}

The $m$-simplicial category presents lower and upper bounds for the $m$-simplicial category.
 
\begin{theorem}
Let $n$ be a natural number, and $K$ be a finite $[\frac{n}{2}]$-connected simplicial complex.
\begin{enumerate}
\item
If  $v_* \in K$, then
\begin{align}
& scat_{n } (K) \le TC_{n} (K,v_*) \le scat_{[\frac{n}{2}]} (K), \label{eqnw1}
\end{align}

\item
If $L\subseteq K$ be a non-empty subcomplex of $K$, then 
\begin{align}
& scat_{n}(K)\leq TC_{n}(K,L) \le TC_{n}(K) \leq scat_{[\frac{n}{2}]}(K\prod K), \label{eqnw3}
\end{align}


\end{enumerate}
\end{theorem}

\begin{proof}
\begin{enumerate}
\item
Let $n$ be an even number, $m = \frac{n}{2}$ and  $scat_{m} (K) = r$. Then there exist $m$-categorical subcomplexes $\Lambda _0,...,\Lambda _r$ of $K$ and $m$-homotopies $F_i :\iota_{\Lambda_i} \simeq c_{v}$ such that $\omega_1 \circ \lambda _i=\iota_{\Lambda _i}$ for some $v \in K$ and for all $0\leq i\leq r$. 
Define $\Omega _i =\Lambda _i \prod \lbrace v_* \rbrace $. It is clear that $\lbrace \Omega _i \rbrace _0 ^r$ covers $K\prod \lbrace  v_*\rbrace$. Put $\sigma _i:\Omega _i\to P_{n}(K,v_*)$ by $\sigma _i (v,v_*)=F_i(v, -) * \gamma_i$ where $\gamma_i$ is the $m$-path from $v$ to $v_*$. Therefore $TC_{n}(K, v_*) \le scat_m(K) = scat_{[\frac{n}{2}]} (K)$.

Now let $TC_{n} (K, v_*) = r$. Then there exist subcomplexes $\Omega_0, \ldots, \Omega_r$ of $K \prod \{v_*\}$ and simplicial maps $\sigma_i: \Omega_i \to P_{n} (K,v_*)$ such that $\pi_{n} \circ \sigma_i = \iota_{\Omega_i}$. Define $\Lambda_i$ as the projection of $\Omega_i$ on the first component and  also define $F_i: \lambda_i \to P_{n} K$ as composition $F_i: \Lambda_i \hookrightarrow \Omega_i \to^{\sigma_i} P_{n}(K, v_*) \hookrightarrow P_{n} K$.  Since $\alpha \circ \sigma_i = \iota_{\Omega_i}$, we have $\alpha \circ F_i = \iota_{\Lambda_i}$ and also $\omega \circ F_i = c_{v_*}$. Therefore $\Lambda_i$ is $n$-categorical and then $scat_{n} (K) \le TC_{n} (K, v_*)$. 

For the odd index, let $m = [\frac{n}{2}]$. Then $n = 2m +1$ and we prove the sequence of inequalities \eqref{eqnw1} in odd cases, 
by using this fact that $TC_{2m +1} \le TC_{2m}$. This yields to the right inequality and the left inequality is obtained similar to the left inequality of 
\eqref{eqnw1}.
%
%
%

\item
Let $n$ be an even number, $m = \frac{n}{2}$ and   $scat_m(K \prod K) = r$. Then there exist $m$-categorical subcomplexes $\Lambda_0, \ldots, \Lambda_r$ of $K \prod K$. Define $\sigma_i (v,v')= p_1 \circ F_i (v, v') * p_2 \circ \overline{F_i (v,v')}$ where $F_i: \Lambda_i \to P_m( K \prod K)$ is the $m$-homotopy between $\iota_{\Lambda_i}$ and the constant map         and $p_1$ and $p_2$ are the projection maps on the first and second components.  
Hence $TC_{n} (K) \le scat_m(K) = scat_{[\frac{n}{2}]} (K)$. 
Now let $TC_{n} (K) =r$. Then similar to the item 1, one can prove that $scat_{n} (K) \le TC_{n} (K,L)$.

Similar to the item (1), for the odd index, again let $m = [\frac{n}{2}]$. Then $n = 2m +1$ and  the sequence of inequalities \eqref{eqnw3} in odd cases, is proved by the inequality  $TC_{2m +1} \le TC_{2m}$.


\end{enumerate}
\end{proof}




\begin{thebibliography}{99}



\bibitem{aghil}
S.A. Aghili, H. Mirebrahimi, and A. Babaee, {\it On the targeted complexity of a map}, Hacettepe Journal of Mathematics and Statistics, vol. 52, no. {\bf 3}, pp. 572--584, 2023.






\bibitem{Lus} 
O. Cornea, G. Lupton, J. Oprea, D.Tanre, \textit{Lusternik-Schnirelmann category},  AMS Monographs 103, 2003.


\bibitem{Inv}
M. Farber, \textit{Invitation to Topological Robotics}, European Mathematical Society, 2008.

\bibitem{Far}
M. Farber, \textit{Topological Complexity of Motion Planning}, Discrete Comput. Geom. {\bf 29}, 211-221, 2003.


\bibitem{Farproj}
M. Farber, S. Tabachnikov, S. Yuzvinsky, {\it Topological robotics: motion planning in projective spaces}, {International Mathematic Research Notices} {\bf 34}, 1853--1870, 2003.


\bibitem{sim}
D. Fernandez-Ternero, J. M. Garcia Calcines, E. Macias-Virgos, J. A. Vilches, \textit{Simplicial fibration} RACSAM {\bf 115} (54), 2021.

\bibitem{dis}
D. Fernandez-Ternero, E. Maccias-Virgos, E. Minuz,  J. A. Vilches, \textit{Discrete topological complexity}, Proc. Amer. Math. Soc. {\bf 146}, 4535-4548, 2018.


\bibitem{Fer}
D. Fernandez-Ternero, E. Macias-Virgos, J. A. Vilches, \textit{Lusternik-Schnirelmann category of simplicial complexes and finite spaces}, Topology and its Application {\bf 194}, 37--50, 2015.





\bibitem{Simp}
J. Gonzalez, \textit{Simplicial complexity: piecewise linear motion planning in robotics}, New York J. Math. {\bf 24}, 279--292, 2008.

\bibitem{Dir}
E. Goubault, M. Farber, A. Sagnier, \textit{Directed topological complexity}, Applied and Computational Topology, hal-02434377, 2019.

\bibitem{Has}
H. Hassanzada, H. Torabi, H. Mirebrahimi and A. Babaee, 
\textit{A Discrete Topological Complexity of Discrete Motion Planning}, Topology and its Applications {\bf 377},
 109634, 2026.




\bibitem{Hat}
A. Hatcher, \textit{Algebraic topology}, Cambridge University Press, Cambridge, 2002. 

\bibitem{Int}
K. D. Joshi, \textit{Introduction to General Topology}, John Wiley \& Sons (Asia) Pte Ltd, 1983. 

\bibitem{com}
D. Kozlov, \textit{Combinatorial algebraic topology}, Algorithms and Computation in Mathematics {\bf 21}, Springer, Berlin, 2008.

\bibitem{Li}
Y. Li, J. Zhao, Z.  Chen, G. Xiong, S. Liu, {\it A Robot Path Planning Method Based on Improved Genetic Algorithm and Improved Dynamic Window Approach}, {Sustainability}, {\bf 15} (5), 46--56, 2023. 


\bibitem{Inva}
W. Lubawski, W. Mazantowicz, \textit{Invariant topological complexity}, Bull. London Math. Soc., {\bf 47}, 101-117, 2014.


\bibitem{Mes}
S. Mescher, Geometric and topological properties of manifolds in robot motion planning, https://arxiv.org/abs/2402.07265, 2024.

\bibitem{5}
A.L. Orekhov, G.L.H. Johnston, N. Simaan, Task and Configuration Space Compliance of Continuum Robots via Lie Group and Modal Shape Formulations, \textit{2023 IEEE/RSJ International Conference on Intelligent Robots and Systems (IROS)}, Detroit, MI, USA, 590-597. 2023.


\bibitem{Pav}
P. Pave{\v{s}}i{\'c}, {\it Topological complexity of a map},
{Homology, Homotopy and Applications}, {\bf 21} (2), 107--130, 2019.



\bibitem{High}
Y. Rudyak, \textit{On higher analogues of topological complexity}, Topology and its applications, {\bf 157}, 916-920, 2010.

\bibitem{Sch}
A.S. Schwarz, \textit{The genus of fiber space}, Math. Sci. {\bf 55}, 44-140, 1996. 

\bibitem{rel}
R. Short, \textit{Relative topological complexity of a pair}, Topology and its Application, 7-23, 2017.

\bibitem{Combi}
K. Tanaka, \textit{ A combinatorial description of topological complexity for finite spaces}, Algebr. Geom. Topol. {\bf 18}, 779-796, 2018.

\bibitem{Teja}
GK. Teja ,PK.  Mohanty ,S. Das, {\it Review on path planning methods for mobile robot}, { Proceedings of the Institution of Mechanical Engineers}, Part C. 2025.








%

\end{thebibliography}
\end{document}